\documentclass[11pt, oneside]{amsart}
 \usepackage[text={5.58in,8.5in},centering,letterpaper,dvips]{geometry}
\usepackage{graphicx}
\usepackage{amsfonts}
\usepackage[dvipsnames]{xcolor}
\usepackage{subcaption}
\usepackage{epsf}
\usepackage{amssymb}
\usepackage{amsmath}
\usepackage{amscd}
\usepackage{tikz-cd}
\usepackage{adjustbox}
\usepackage{pdfpages}
\usepackage{fancyhdr}
\usepackage{setspace}
\usepackage{soul} 
\usepackage[all]{xy}
\usepackage{verbatim}
\usepackage{enumerate}
\usepackage[colorlinks=true, urlcolor=NavyBlue, linkcolor=NavyBlue, citecolor=NavyBlue,backref=page]{hyperref}
\usepackage{mdframed}
\usepackage{wrapfig}

\renewcommand*{\backref}[1]{}
\renewcommand*{\backrefalt}[4]
{%
    \ifcase #1 (Not cited.)%
        \or        (Cited on page~#2.)
        \else      (Cited on pages~#2.)
    \fi
}

\makeatletter
\renewcommand{\tocsection}[3]{%
  \indentlabel{\@ifnotempty{#2}{\bfseries\ignorespaces#1 #2\quad}}\bfseries#3}
\renewcommand{\tocsubsection}[3]{%
  \indentlabel{\@ifnotempty{#2}{\ignorespaces#1 #2\quad}}#3}

\newcommand\@dotsep{4.5}
\def\@tocline#1#2#3#4#5#6#7{\relax
  \ifnum #1>\c@tocdepth 
  \else
    \par \addpenalty\@secpenalty\addvspace{#2}%
    \begingroup \hyphenpenalty\@M
    \@ifempty{#4}{%
      \@tempdima\csname r@tocindent\number#1\endcsname\relax
    }{%
      \@tempdima#4\relax
    }%
    \parindent\z@ \leftskip#3\relax \advance\leftskip\@tempdima\relax
    \rightskip\@pnumwidth plus1em \parfillskip-\@pnumwidth
    #5\leavevmode\hskip-\@tempdima{#6}\nobreak
    \leaders\hbox{$\m@th\mkern \@dotsep mu\hbox{.}\mkern \@dotsep mu$}\hfill
    \nobreak
    \hbox to\@pnumwidth{\@tocpagenum{\ifnum#1=1\bfseries\fi#7}}\par
    \nobreak
    \endgroup
  \fi}
\AtBeginDocument{%
\expandafter\renewcommand\csname r@tocindent0\endcsname{0pt}
}
\def\l@subsection{\@tocline{2}{1pt 
}{5pc 
}{}{}}
\makeatother

\theoremstyle{theorem}
\newtheorem{theorem}{Theorem}[section]
\newtheorem{proposition}[theorem]{Proposition}
\newtheorem{lemma}[theorem]{Lemma}
\newtheorem{question}[theorem]{Question}
\newtheorem{corollary}[theorem]{Corollary}
\newtheorem{conjecture}[theorem]{Conjecture}

\newtheorem*{namedthm}{\namedthmname}
\newcounter{namedthm}

\makeatletter
\newenvironment{named}[1]
  {\def\namedthmname{#1}%
   \refstepcounter{namedthm}%
   \namedthm\def\@currentlabel{#1}}
  {\endnamedthm}
\makeatother

\makeatletter
\newcommand{\newreptheorem}[2]{%
\newtheorem*{rep@#1}{\rep@title}
\newenvironment{rep#1}[1]{%
\def\rep@title{#2 \ref{##1}}
\begin{rep@#1}}{\end{rep@#1}}}
\makeatother

\newreptheorem{theorem}{Theorem}
\newreptheorem{lemma}{Lemma}
\newreptheorem{question}{Question}
\newreptheorem{corollary}{Corollary}
\newreptheorem{proposition}{Proposition}
\newreptheorem{conjecture}{Conjecture}

\theoremstyle{definition}
\newtheorem{definition}[theorem]{Definition}
\newtheorem{remark}[theorem]{Remark}

\makeatletter
\newcommand{\newreptheoremD}[2]{%
\newtheorem*{rep@#1}{\rep@title}
\newenvironment{rep#1}[1]{%
\def\rep@title{#2 \ref{##1}}
\begin{rep@#1}}{\end{rep@#1}}}
\makeatother

\newreptheoremD{definition}{Defintion}

\newcommand{\Z}{\mathbb{Z}}

\newcommand{\R}{\mathbb{R}}

\newcommand{\Id}{\text{Id}}
\newcommand{\Int}{\text{Int}}

\newcommand{\Bb}{\mathcal B}
\newcommand{\Cc}{\mathcal C}
\newcommand{\Dd}{\mathcal D}

\newcommand{\Ff}{\mathcal F}

\newcommand{\Pp}{\mathcal P}

\newcommand{\Ss}{\mathcal S}

\newcommand{\Diff}{\text{Diff}}
\newcommand{\Sing}{\text{Sing}}
\newcommand{\Fix}{\text{Fix}}
\newcommand{\Span}{\text{Span}}
\newcommand{\Stab}{\text{Stab}}

\newcommand{\Isom}{\text{Isom}}

\makeatletter
\def\@seccntformat#1{%
  \protect\textup{\protect\@secnumfont
    \ifnum\pdfstrcmp{subsection}{#1}=0 \bfseries\fi
    \csname the#1\endcsname
    \protect\@secnumpunct
  }%
}  
\makeatother

\begin{document}

\rhead{\thepage}
\lhead{\author}
\thispagestyle{empty}


\raggedbottom
\pagenumbering{arabic}
\setcounter{section}{0}


\title{An equivariant Laudenbach-Po\'enaru theorem}

\author{Jeffrey Meier}
\address{Department of Mathematics, Western Washington University 
Bellingham, WA 98229}
\email{jeffrey.meier@wwu.edu}
\urladdr{http://jeffreymeier.org} 

\author{Evan Scott}
\address{Department of Mathematics, CUNY Graduate Center
New York, NY 10016}
\email{escott@gradcenter.cuny.edu}
\urladdr{https://sites.google.com/view/evanscott}

\begin{abstract}
	A foundational theorem of Laudenbach and Po\'enaru states that any diffeomorphism of $\#^n(S^1\times S^2)$ extends to a diffeomorphism of $\natural^n(S^1\times B^3)$.
	We prove a generalization of this theorem that accounts for the presence of a finite group action on $\#^n(S^1\times S^2)$.
	Our proof is independent of the classical theorem, so by considering the trivial group action, we give a new proof of the classical theorem.
	
	Specifically, we show that any finite group action on $\#^n(S^1\times S^2)$ extends to a \emph{linearly parted} action on $\natural^n(S^1\times B^3)$ and that any two such extensions are equivariantly diffeomorphic.
	Roughly, a linearly parted action respects a decomposition into equivariant $0$--handles and $1$--handles, where, for each handle in the decomposition, its stabilizer acts linearly on that handle.
	The restriction to linearly parted actions is important, because there are infinitely many distinct nonlinear actions on $B^4$ with identical actions on $\partial B^4$; these nonlinear actions give extensions of the same action on $\partial B^4$ which are \emph{not} equivariantly diffeomorphic.
		
	We also prove a more general theorem: Every finite group action on $\left(\#^n(S^1\times S^2),L\right)$, with $L$ an invariant unlink, extends across a pair $\left(\natural^n(S^1\times B^3),\Dd\right)$, with $\Dd$ an equivariantly boundary-parallel disk-tangle, and any two such extensions are equivariantly diffeomorphic.
\end{abstract}

\maketitle


\section{Introduction}
\label{sec:introduction}

A foundational theorem of Laudenbach and Po\'enaru states that any diffeomorphism of $\#^n(S^1\times S^2)$ extends to a diffeomorphism of $\natural^n(S^1\times B^3)$~\cite{LauPoe_72_A-note-on-4-dimensional-handlebodies}.
Modern diagrammatic approaches to four-manifold topology, such as Kirby diagrams~\cite[Chapter~4]{GomSti_99_4-manifolds-and-Kirby} and trisection diagrams~\cite{GayKir_16_Trisecting-4-manifolds}, rely on this theorem.
In particular, this theorem implies that in a handle-decomposition of a closed $4$--manifold, the $3$--handles and $4$--handles are uniquely determined by the $0$--handles, $1$--handles, and $2$--handles.

In this paper, we prove an analog of this theorem in the presence of a finite group action.
In order to do so, we first need to identify the proper equivariant analog of a 4--dimensional 1--handlebody.
Motivated by the~\ref{LT} for actions in dimension three, we consider a class of finite group actions on 4--dimensional 1--handlebodies that we refer to as \emph{linearly parted.}
Roughly, these actions decompose equivariantly into $0$--handles and $1$--handles such that the  induced action of the stabilizer of each handle is linear on that handle; see Definition~\ref{def:linearly_parted} for precise details.
We prove that a version of the Laudenbach-Po\'enaru theorem holds for these actions.

\begin{reptheorem}{thm:laud_poen}
	Let $G$ be a finite group acting on $Y=\#^k(S^1\times S^2)$.
	\begin{enumerate}
		\item There exists a 4--dimensional 1--handlebody $X$ with $\partial X = Y$ such that the action of $G$ on $Y$ extends to a linearly parted action on $X$.
		\item If $X$ and $X'$ are two 4--dimensional 1--handlebodies with $\partial X = Y = \partial X'$ and both $X$ and $X'$ have linearly parted $G$--actions extending the $G$--action on $Y$, then $X$ and $X'$ are $G$--diffeomorphic rel-boundary.
	\end{enumerate}
\end{reptheorem}

This theorem statement is equivalent to the usual statement about extending diffeomorphisms; see Proposition~\ref{prop:equivalent_statements}.
Our proof of Theorem~\ref{thm:laud_poen} is independent of the classical Laudenbach-Po\'enaru theorem, so we obtain a new proof of the classical Laudenbach-Po\'enaru theorem by setting $G=1$.
Another independent new proof of the classical theorem, using completely different techniques, was recently given by Moussard~\cite{Mou_25_On-diffeomorphisms-of-4-dimensional-1-handlebodies}.

Our approach utilizes a combination of standard 3--dimensional techniques (adapted to the equivariant setting), together with recent equivariant versions of the theorems of Munkres and Cerf regarding the connectedness of $\Diff^+(S^n)$ for $n=2,3$~\cite{Mun_60_Differentiable-isotopies-on-the-2-sphere, Cer_68_Sur-les-diffeomorphismes-de-la-sphere-de-dimension} given by Dinkelbach and Leeb~\cite[Corollary~2.6]{DinLee_09_Equivariant_Ricci_flow} and Mecchia and Seppi~\cite{MecSep_19_Isometry-groups-and-mapping}, respectively; see also~\cite[Corollary~1.5]{ChoLi_24_Equivariant-3-manifolds-with}.

Our restriction to considering linearly parted actions is critical:
The negative resolution of the Smith Conjecture in dimension four~\cite{Gif_66_The-generalized-Smith-conjecture, Gor_74_On-the-higher-dimensional-Smith} shows that there are many actions on $B^4$ whose restrictions to $\partial B^4$ are identical to a linear action, but which differ pairwise on the interior.
Such actions cannot satisfy Theorem~\ref{thm:laud_poen}(2), since there is no equivariant diffeomorphism (rel-boundary or otherwise) between such an action and a linear one.
Restricting to linearly parted actions eliminates these issues.

We have a number of motivations for proving Theorem~\ref{thm:laud_poen}.
First, we view this theorem as a first step towards understanding a number of open questions about group actions in dimension four; see Section~\ref{sec:open_questions} for a discussion.
Second, just as the classical Laudenbach-Po\'enaru theorem facilitates the study 4--manifolds using handle-decompositions and trisections, Theorem~\ref{thm:laud_poen} enables the development of equivariant versions of these techniques.

Following these lines, in a forthcoming paper~\cite{MeiSco_tri}, we develop a theory of equivariant trisections that can be used to study finite group actions on four-manifolds.
In that upcoming work, we also show that equivariant Kirby diagrams for $4$--manifolds always exist~\cite[Section~4.3]{MeiSco_tri}, but we leave full development of an equivariant Kirby calculus to further work.

A central aspect of the theory of trisections is its adaptation to the setting of knotted surfaces via bridge trisections~\cite{MeiZup_17_Bridge-trisections,MeiZup_18_Bridge-trisections}.
The theory of bridge trisections relies on a result~\cite[Lemma~18]{MeiZup_18_Bridge-trisections} that generalizes the classical Laudenbach-Po\'enaru theorem to pairs $(X,\Dd)$, where $X\cong \natural^n(S^1\times B^3)$ and $\Dd\subset X$ is a boundary-parallel disk-tangle.
With an eye toward the development of an equivariant version of the theory of bridge trisections in~\cite{MeiSco_tri}, we prove the following equivariant version of the Laudenbach-Po\'enaru theorem for pairs $(X,\Dd)$.

\begin{reptheorem}{thm:equiv_patch}
	Let $G$ be a finite group acting on $Y=\#^k(S^1\times S^2)$, and let $L\subset Y$ be a $G$--invariant unlink.
	\begin{enumerate}
		\item There exists a $G$--equivariant filling $(X,\Dd)$ of $(Y,L)$.
		\item If $(X,\Dd)$ and $(X',\Dd')$ are two $G$--equivariant fillings of $(Y,L)$, then $(X,\Dd)$ and $(X',\Dd')$ are $G$--diffeomorphic rel-boundary.
	\end{enumerate}
\end{reptheorem}

As in Theorem~\ref{thm:laud_poen}, this theorem relies on identifying the proper equivariant analog of the pair $(X,\Dd)$.
Roughly, the action on $X$ should be linearly parted, as before, and $\Dd$ should be equivariantly boundary-parallel in a way that respects the linear parting; see Definition~\ref{def:linearly_parted_pair} for precise details.

\subsection*{Organization}

In Section~\ref{sec:background}, we briefly recall some background material on finite group actions.
In Section~\ref{sec:equivariant_handlebodies}, we discuss the equivariant 1--handlebodies that play a central role throughout the paper.
We also discuss some important results from 3--dimensional equivariant topology.
In Section~\ref{sec:ELP}, we prove Theorem~\ref{thm:laud_poen}.
In Section~\ref{sec:ELP_surfaces}, we discuss how to adapt the concepts and techniques from Sections~\ref{sec:equivariant_handlebodies} and~\ref{sec:ELP} to the setting of actions on pairs $(X,\Dd)$ and we prove Theorem~\ref{thm:equiv_patch}.

\subsection*{Acknowledgements}

The authors wish to thank David Gay, Daniel Hartman, Malcolm Gabbard, and Maggie Miller for helpful conversations.
The first author was supported by NSF grants DMS-2006029 and DMS-2405324.

\section{Background on finite group actions}
\label{sec:background}

In this section, we recall some basic definitions regarding group actions that will be used extensively in what follows.
We refer the reader to~\cite{Bre_72_Introduction-to-compact-transformation-groups} for a thorough treatment with complete details.
Throughout the paper, $G$ will denote a finite group, and all group actions considered will be faithful, smooth, and properly discontinuous.
Let $M$ be a manifold, and let $G$ act on $M$.
We will always assume that every element of $G$ acts on $M$ as an orientation-preserving diffeomorphism, i.e. the group acts \emph{orientation-preservingly}, unless explicitly indicated otherwise.
We encode this set-up by saying that $M$ is a \emph{$G$--manifold}.

We identify a group action $\rho\colon G\to \Diff(M)$ with its image in $\Diff(M)$, so that $\rho$ is identified with $\rho\circ\phi$ for any automorphism $\phi$ of $G$. 
Applying $\phi$ amounts to relabelling the elements in the image of $\rho$ and will not be regarded as a distinct action.
Since we assume all actions of $G$ are faithful, we assume $\rho$ is injective.
Two actions $\rho_1$ and $\rho_2$ of $G$ on $M$ are \emph{(smoothly) equivalent} if there exists $\psi\in \Diff(M)$ so that $\psi^{-1}\rho_1(g)\psi = \rho_2(g)$ for all $g\in G$, i.e. the two actions are conjugate in $\Diff(M)$\footnote{
Again, strictly speaking we ask that the images are conjugate subgroups by $\psi$, but we can always choose isomorphisms $\rho_1$ and $\rho_2$ of $G$ onto these two subgroups which satisfy the above equation on the nose.}.
This coincides with the notion that $(M, \rho_1)$ and $(M, \rho_2)$ are equivariantly diffeomorphic by a diffeomorphism $\psi$.

A submanifold $N\subseteq M$ is \emph{invariant} if $g\cdot N = N$ for all $g\in G$.
An invariant submanifold always admits an \emph{open invariant tubular neighborhood} $\nu(N)$, which is a (smooth) $G$--vector bundle over $N$; any such open invariant tubular neighborhood contains a closed invariant tubular neighborhood~\cite[Chapter~VI, Theorem~2.2]{Bre_72_Introduction-to-compact-transformation-groups}.
A submanifold $N\subseteq M$ is \emph{equivariant} if, for all $g\in G$, either $g\cdot N = N$ or $(g\cdot N)\cap N = \varnothing$.
Clearly all invariant submanifolds are equivariant.

Given an equivariant submanifold $N\subseteq M$, the \emph{stabilizer} of $N$ is the subgroup of $G$ consisting of those $g\in G$ such that $g\cdot N = N$ and is denoted $\Stab_G(N)$.
The \emph{point-wise stabilizer} of $N$ is the set of elements $g\in G$ that fix each point in $N$.
Note that the point-wise stabilizer of $N$ is normal in the set-wise stabilizer of $N$; we denote the corresponding quotient by $G_N$.
The quotient group $G_N$ acts on $N$, and we refer to this action as the \emph{induced action} of $G$ on $N$.
We say that $G$ acts on $N$ with some property (say, orientation-preservingly) if $G_N$ acts on $N$ with that property.
Note that if $M$ is connected and $N$ is codimension--$0$ (such as when $N$ is a tubular neighborhood, as will often be the case) then $G_N = \Stab_G(N)$: if $g\in G$ fixes a codimension--$0$ submanifold of connected manifold $M$ point-wise, then $g$ fixes all of $M$ and is hence trivial by faithfulness.
Therefore, for orientation-preserving $G$--actions, $G$ always acts on $\nu(N)$ orientation-preservingly, even when $G$ does not act orientation-preservingly on $N$.

\subsection{Linear actions}
\label{subsec:linear}

A central aspect of this work is the idea of breaking up $4$--dimensional $G$--manifolds into simple pieces for which the induced actions can be understood as amalgamations of linear actions.

\begin{definition}
	A $G$--action on $B^n$ is \emph{linear} if it is smoothly equivalent to the action of an $n$--dimensional, orthogonal representation of $G$ on $B^n$.
	A $G$--action on $S^{n-1}$ is \emph{linear} if it is the boundary of a linear $G$--action on $B^n$.
\end{definition}

The following important observation requires a quick definition.
Given a $G$--action on $M$, the \emph{equivariant cone} of the action is the $G$--action on the cone on $M$ given by $g\cdot(x,t) = (g\cdot x,t)$ for all $g\in G$ and each $(x,t)\in (M\times[0,1])/\{(y,0)\sim(y',0)\}$.

\begin{lemma}
\label{lem:balls_as_cones}
	A linear $G$--action on $B^n$ is the equivariant cone of the induced $G$--action on $\partial B^n$, with the origin as cone point.
	Furthermore, if the action on $B^n$ has a fixed point $p$ in $\partial B^n$, then the action on $B^n$ is the equivariant cone of the induced action on an invariant $(n-1)$--ball in $\partial B^n$ not containing $p$.
\end{lemma}

\begin{proof}
	The first claim follows from the fact that $G$ preserves an inner product, hence the length of every vector.
	Thus, $G$ leaves each sphere of radius $0<r\leq 1$ invariant and fixes the origin.
	Since $G$ is linear, the action on each sphere of radius $r$ is the same, so the action is a cone as desired.

	For the second claim, see Figure~\ref{fig:balls_as_cones}.
	Since $p$ is fixed, the span of $ p$ is fixed, and each $(n-1)$--ball orthogonal to this span is invariant.
	Thus, the $G$--action on $B^n$ is equivalent to a product $G$--action on
	$$\Span\{ p\}\oplus \Span\{ p\}^\perp \cong [-1,1]\times B^{n-1},$$
	with $G$ acting trivially on the interval factor.
	
	The interval can be regarded as the cone on $\{-p\}$ with cone point $p$, and each $B^{n-1}$ can be regarded as the cone on its boundary $(n-2)$--sphere with cone point its intersection with the span of $p$.
	For $t\in(0,1]$, let
	$$M_t = (\{p(1-2t)\}\times B_t^{n-1})\cup(\partial B_t^{n-1}\times[p(1-2t),p]),$$
	as depicted in Figure~\ref{fig:balls_as_cones}.
	Then, each $M_t$ is $G$--invariant, and $B^n$ is the cone on $M_1$ with cone point $p$.
	So, the $G$--action on $B^n$ is an equivariant cone on the induced action on the $(n-1)$--ball $M_1$, as desired.
\end{proof}

\begin{figure}[ht!]
	\centering
	\includegraphics[width = .6\linewidth]{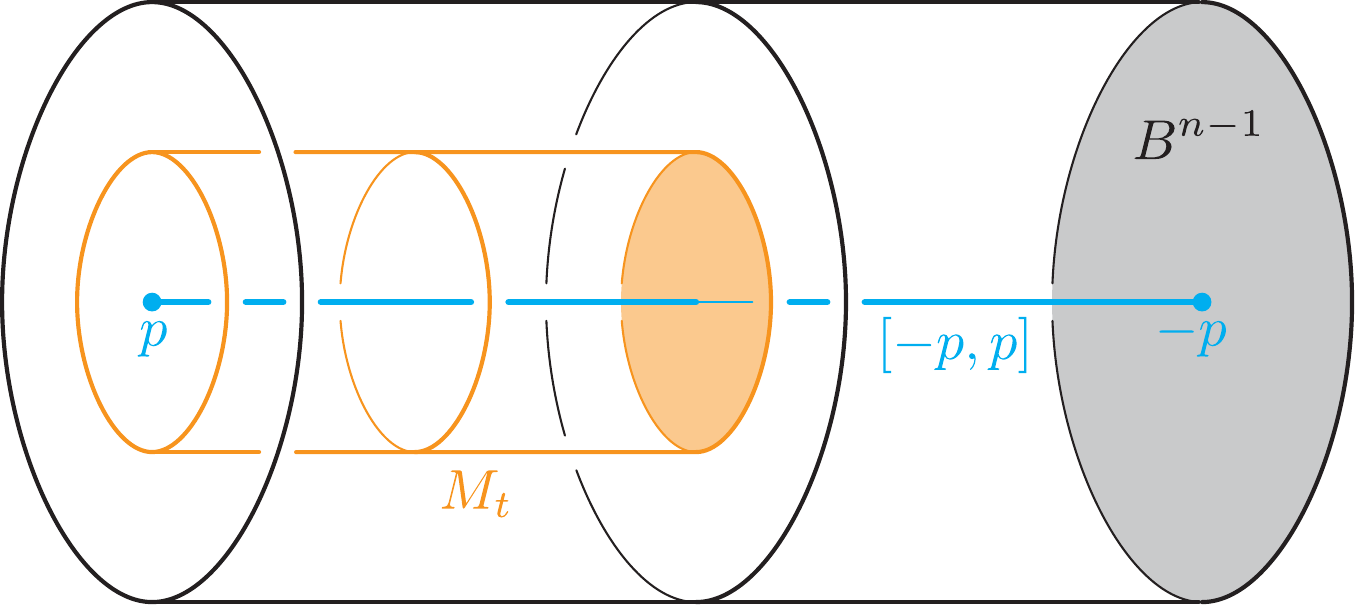}
    \caption{An action on $B^n$ fixing a point $p$ in the boundary is the equivariant cone on an invariant $B^{n-1}$ in its boundary with cone point $p$.}
    \label{fig:balls_as_cones}
\end{figure}

For actions that are equivariant cones, we have the following.

\begin{lemma}
\label{lem:cone_equivalence}
	Suppose $G$ acts on $M$ as the equivariant cone of a $G$--action on $N\subseteq M$ with cone point $p$.
	Suppose $A\subseteq M$ is any $G$--invariant neighborhood of $p$ such that the induced $G$--action on $A$ is an equivariant cone with cone point $p$.
	Then $A$ and $M$ are equivariantly diffeomorphic via a diffeomorphism that is the identity on some neighborhood of $p$ contained in both $A$ and $M$.
\end{lemma}

\begin{proof}
	This follows from the proofs of Lemma~8.2 and Theorem~8.3 of~\cite[Chapter~II]{Bre_72_Introduction-to-compact-transformation-groups}, together with Schwarz's extension to the smooth setting of Palais'  covering homotopy theorem~\cite{Pal_60_The-classification-of-G-spaces, Sch_80_Lifting-smooth-homotopies}.
\end{proof}

Note that if a $G$--action on $M$ is an equivariant cone, then the cone point is a fixed point.
Every fixed point $p$ has a closed invariant tubular neighborhood $N$ for which the induced action $G_N$ is linear. 
If $p\in\partial M$, then $G_N$ fixes a point in $\partial N$, so this induced action is an equivariant cone in the second way described in Lemma~\ref{lem:cone_equivalence}.
Mirroring~\cite[Chapter II, Section~8]{Bre_72_Introduction-to-compact-transformation-groups}, this leads to the following characterization of equivariant cones.

\begin{corollary}
\label{cor:Cones_Linear}
	Suppose $G$ acts on $M$ as the equivariant cone of a $G$--action on $N\subseteq M$ with cone point $p$.
	Then $M$ is equivariantly diffeomorphic to a linear action on $B^n$, and either the action on $N$ is equivariantly diffeomorphic to a linear action on $S^{n-1} = \partial B^n$, which occurs when $p$ is in the interior of $M$, or the action on $N$ is equivariantly diffeomorphic to a linear action on $B^{n-1}\subset \partial B^n$, which occurs when $p$ is in the boundary of $M$.
\end{corollary}

The above discussion establishes the following analogy:
\emph{The analog of the ball $B^n$ (respectively, the sphere $S^n$) in the equivariant category is a \emph{linear} action on $B^n$ (respectively, a linear action on $S^n$.)}

If we take the equivariant cone of a nonlinear $G$--action on $S^n$, we get a non-smooth action on $B^{n+1}$, with the failure of smoothness occurring at the cone point.
The above corollary shows that for any real representation of $G$ acting on $\R^n$, the induced action on any star-like invariant ball containing the origin is smoothly equivalent to a linear action.
This is a nice criterion for linearity that immediately implies that the suspension of a linear action is linear, among other consequences.

\section{Equivariant $1$--handlebodies}
\label{sec:equivariant_handlebodies}

The classical Laudenbach and Po\'enaru Theorem concerns fillings of $\#^n(S^1\times S^2)$ by $4$--dimensional $1$--handlebodies $\natural^n(S^1\times B^3)$.
To generalize this theorem to the equivariant setting, we need to identify the correct analog to a $1$--handlebody in the equivariant category.
More precisely, we need to identify a sufficiently well-behaived class of actions on $1$--handlebodies.
In Section~\ref{subsec:linear}, we showed that the correct equivariant analog to the ball $B^n$ is the ball $B^n$ equipped with a linear action.
The equivariant analog to a $1$--handlebody will be built out of these linear actions on balls.

\begin{definition}
\label{def:linearly_parted}
	Let $H$ be an $n$--dimensional 1--handlebody.
	A \emph{ball-system} for $H$ is a disjoint union $\Bb$ of neatly embedded\footnote{
	A submanifold $N$ is \emph{neatly embedded} in $M$ if $\partial N = N\cap\partial M$ and $N$ is transverse to $\partial M$~\cite[Chapter~1.4]{Hir_94_Differential-topology}.
	}
	$(n-1)$--balls in $H$ such that $H\setminus\nu(\Bb)$ is a disjoint union of $n$--balls.
	We say that a $G$--action on $H$ is \emph{parted} if there exists a $G$--invariant ball-system $\Bb$ for $H$.
	We say that a $G$--action on $H$ is \emph{linearly parted} if it is parted by $\Bb$, and the action of $G$ on each $(n-1)$--ball of $\Bb$ and on each $n$--ball component of $H\setminus\Bb$ is linear.
\end{definition}

In other words, the $G$--manifold $H$ is linearly parted if and only if it admits a decomposition into equivariant $0$--handles and $1$--handles such that the induced action on each handle is linear.
For this reason, we refer $H$ as a \emph{$G$--equivariant} $1$--handlebody.

\begin{remark}
\label{rmk:automatic_linear}
	In dimensions four and below, the condition of linearity for the induced $G$--action on $\Bb$ is automatic, though this requires deep theorems from $3$--manifold topology (cf. the~\ref{LT} below). We nonetheless state the definition in arbitrary dimension for completeness.
\end{remark}

For complete details on the handlebody perspective, we refer the reader to Wasserman's development of equivariant Morse theory and handlebody theory~\cite{Was_69_Equivariant-differential-topology}.
Nonetheless, the following definition will be useful throughout this work.

\begin{definition}
\label{def:equivariant_handle}
A \textit{$G$--equivariant $n$--dimensional $k$--handle bundle} $\frak H$ is a $G$--manifold diffeomorphic to a disjoint union of $n$--balls $D^{k}\times D^{n-k}$, so that the $G$--action on $\frak H$ is transitive on components and the stabilizer of each component $\frak h\subseteq \frak H$ (which we refer to as an \emph{equivariant handle}) acts on $\frak h \cong D^{k}\times D^{n-k}$ preserving the product and linearly in each factor.
\end{definition}

It is clear that, given an equivariant $(k-1)$--sphere $S$ in $\partial X$ for a $G$--manifold $X$, after choosing an equivariant diffeomorphism from $\partial D^k\times D^{n-k}$ to an equivariant regular neighborhood of $S$, we may attach $\frak H$ along the orbit of $S$ equivariantly to get a new $G$--manifold $X'$, the result of $k$--handle attachment to $X$ along $S$.
We will sometimes refer to $X'$ as $X\cup \frak h$, where $\frak h\subseteq \frak H$ is attached along $S$, and say $X'$ is the result of equivariant $k$--handle attachment to $X$, where it is understood that $X'$ is $X$ union the handle-bundle $\frak H$, the orbit of $\frak h$.
We'll make use of this terminology in Lemmata~\ref{lem:unique_handle_attachment} and~\ref{lem:2_handle_attachment}.

\subsection{Linear actions in dimension three}
\label{subsec:dim3}

To motivate Definition~\ref{def:linearly_parted}, we now briefly review important results from $3$--dimensional equivariant topology that are central to the present paper.
First, we have the Equivariant Loop Theorem of Meeks and Yau~\cite{MeeSimYau_82_Embedded-minimal,MeeYau_79_The-classical-Plateau,MeeYau_80_Topology-of-three-dimensional}, given here for reference as stated by Edmonds~\cite{Edm_86_A-topological-proof}, who gave a topological proof.

\stepcounter{theorem}
\begin{named}{Equivariant Loop Theorem~\thetheorem}
\label{ELT}
	Let $G$ act on a $3$--manifold $M$.
	Let $C\subset\partial M$ be a simple closed curve such that
	\begin{enumerate}
		\item $C$ is null-homotopic in $M$,
		\item $C$ is $G$--equivariant, and
		\item $C$ is transverse to the exceptional set $E$ of the action of $G$ on $\partial M$.
	\end{enumerate}
	Then there is a $G$--equivariant, neatly embedded disk $D\subset M$ with $\partial D=C$.
\end{named}

The next result is a consequence of the resolutions of the Spherical Space Form Conjecture~\cite{Per_02_The-entropy-formula,Per_03_Finite-extinction,Per_03_Ricci-flow}) and the Smith Conjecture \cite{MorBas_84_The-Smith-conjecture}.

\stepcounter{theorem}
\begin{named}{Linearization Theorem~\thetheorem}
\label{LT}
	Every group action on $S^3$ or $B^3$ is linear.
\end{named}

Together, these two major theorems imply that every action on a 3--dimensional 1--handlebody is linearly parted; cf. work of McCullough, Miller, and Zimmerman that develops a general theory for studying actions on 3--dimensional 1--handlebodies~\cite{McCMilZim_89_Group-actions-on-handlebodies}.

\begin{corollary}
\label{cor:3d_parted}
	Let $H$ be a 3--dimensional 1--handlebody, and let $G$ act on $H$.
	There exists an equivariant disk-system $\Dd$ for $H$ such that the action of $G$ on each 3--ball of $H\setminus\Dd$ is equivalent to a linear action.
\end{corollary}

In light of this, we can conceptualize the notion of a linearly parted action on an $n$--dimensional 1--handlebody as an analog of the well-behaved nature of actions on 3--dimensional 1--handlebodies. 

\subsection{Linearly parted actions on $B^4$}
\label{subsec:B4action}

To our knowledge, it is not known whether a $G$--action on a 4--dimensional 1--handlebody is necessarily parted (in constrast to the above), and open questions along this line are, in our view, quite interesting; see Section~\ref{sec:open_questions} for more discussion.
We can say, however, that there are $G$--actions on 4--dimensional 1--handlebodies that are not \emph{\textbf{linearly}} parted; see Corollary~\ref{cor:not_parted_B4}.
This is a consequence of the following lemma, which establishes a condition for equivariant handle-cancellation and implies (see Corollary~\ref{cor:parted_Bn}) that linearly parted actions on $B^4$ are linear.

\begin{lemma}
\label{lem:cancel_parting}
	Let $G$ act on an $n$--dimensional 1--handlebody $H$ so that the action is linearly parted by a ball-system $\Bb$.
	Let $Z_1$ and $Z_2$ be two 0--handles that are in different $G$--orbits, and let $B = Z_1\cap Z_2$ be a parting ball in $\Bb$.
	If $B$ is invariant with respect to the action of $\Stab_G(Z_2)$, then $\Bb\setminus(G\cdot B)$ is a linear parting system for $H$.
	In other words the triple $(Z_1, B, Z_2)$ can be replaced with a single 0--handle $Z = Z_1\cup_B Z_2$, cancelling $Z_2$ and $B$, and this can be extended equivariantly to the orbit of $(Z_1, B, Z_2)$.
\end{lemma}

An example of the set-up for this lemma is shown in Figure~\ref{fig:ring_of_balls}.

\begin{proof}
	We first show that the action on $Z_1\cup_BZ_2$ by $\Stab_G(Z_2)$ is linear.
	By hypothesis, $B$ is $\Stab_G(Z_2)$--invariant, and the induced $\Stab_G(Z_2)$--action on $B$ is linear.\footnote{In this paper, $n\leq 4$, so the induced $\Stab_G(Z_2)$--action on $B$ is automatically linear; see Remark~\ref{rmk:automatic_linear}.}
	Thus, there is a fixed point $p\in B$.
	
	Taking an invariant tubular neighborhood $\nu(B)$ of $B$ in $Z_1\cup_B\cup Z_2$, we see that $\nu(B)\cap Z_2$ is diffeomorphic to a product $B\times [0,1]$ with $B$ being $B\times\{0\}$.
	Hence, the induced action of $G$ on $\nu(B)\cap Z_2$ is an equivariant cone with cone point $p$, as is the induced action of $G$ on $Z_2$, by Lemma~\ref{lem:balls_as_cones}.

	Contracting $B$ onto $\nu(p)$, we have that $Z=Z_1\cup_B Z_2$ is $\Stab_G(Z_2)$--equivariantly diffeomorphic to $Z_1\cup_{\nu(p)}Z_2$, by Lemma~\ref{lem:cone_equivalence}.
	Again by Lemma~\ref{lem:cone_equivalence}, we can contract $Z_2$ onto $(\nu(B)\cap Z_2)$, that is, $Z_1\cup_{\nu(p)}Z_2$ is $\Stab_G(Z_2)$--equivariantly diffeomorphic to $Z_1\cup_{\nu(p)}(\nu(B)\cap Z_2)$, and since $\nu(B)\cap Z_2$ is equivariantly diffeomorphic to $B\times[0,1]$, this union is equivariantly diffeomorphic to $Z_1$, so the action on $Z = Z_1\cup_B Z_2$ is linear, as is the action on $Z_1$.
	
	Thus, replacing $(Z_1,B,Z_2)$ with $Z$ gives rise to a new linear parting with respect to $\Stab_G(Z_2)$.
	Since $Z_2$ is an equivariant submanifold, and since $Z_1$ and $Z_2$ are in different $G$--orbits, we can apply the same transformation along the entire $G$--orbit of $B$ to get a new linear parting for $H$.
\end{proof}

\begin{wrapfigure}[16]{r}{0.42\textwidth}
\vspace{-5mm}
\begin{mdframed}
	\centering
	\includegraphics[width=\textwidth]{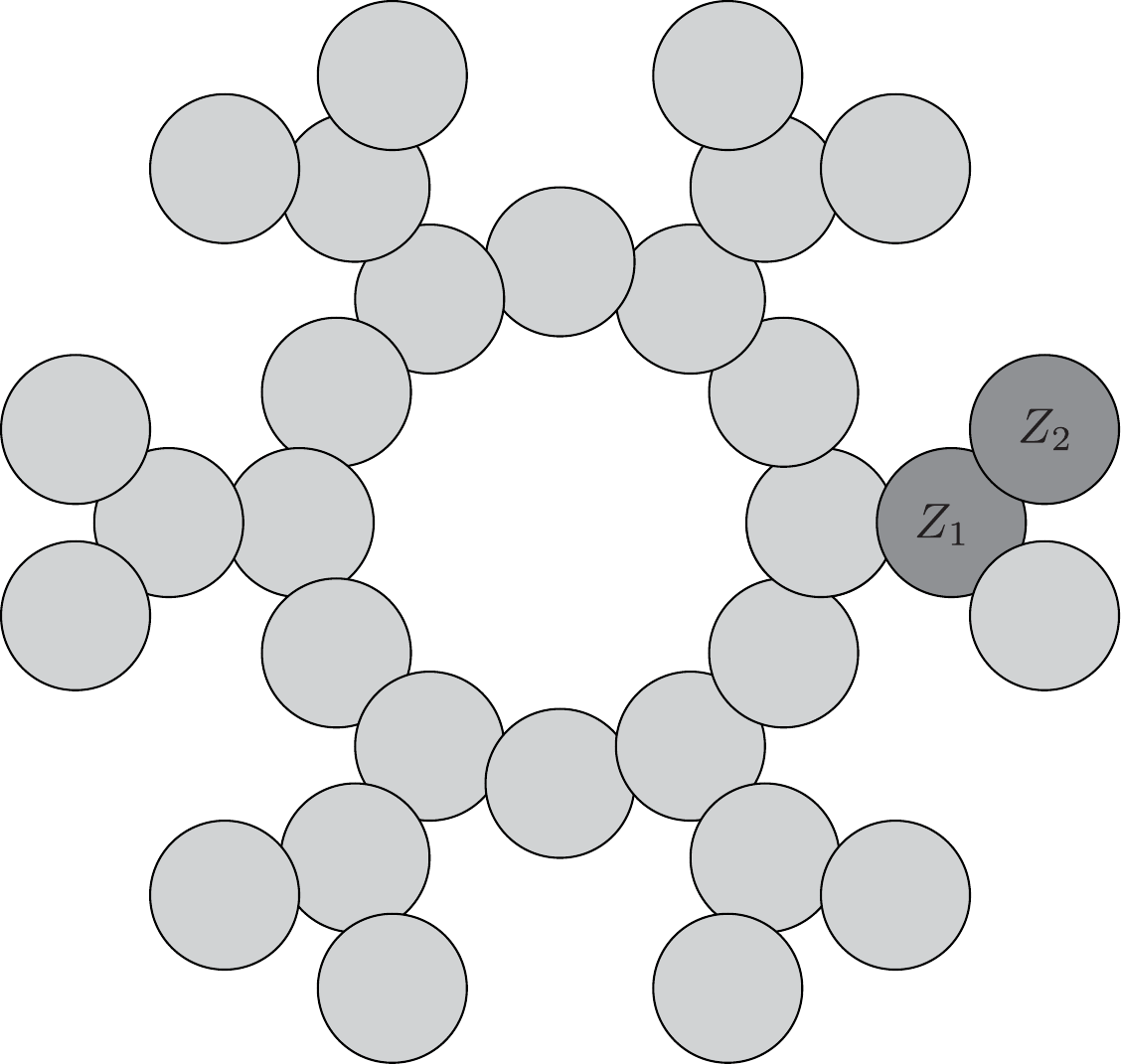}
	\caption{A solid torus, linearly parted into 3--balls with respect to the obvious, cyclic $\Z_6$--action}
	\label{fig:ring_of_balls}
\end{mdframed}
\end{wrapfigure}

In contrast to the above lemma, it is easy to see that there can be pairs of equivariant 0--handles and 1--handles that cancel topologically but not equivariantly: For example, the linear parting illustrated in Figure~\ref{fig:ring_of_balls} can be simplified using Lemma~\ref{lem:cancel_parting} until it is a ring of six balls, but it cannot be simplified further.
If the action in Figure~\ref{fig:ring_of_balls} is promoted to a $D_6$--action, the linear parting can still be reduced to a ring of six balls, but more care is needed.

An immediate consequence is a characterization of linear partings of actions on balls.

\begin{corollary}
\label{cor:parted_Bn}
	A $G$--action on $B^n$ is linearly parted if and only if it is linear.
\end{corollary}

\begin{proof}
	The reverse direction is trivial: linear actions are automatically linearly parted by the empty ball-system.
	
	Conversely, suppose the action of $G$ on $B^n$ is linearly parted.
	If necessary, modify the linear parting system so that every parting ball has orientation-preserving stabilizer, i.e. so that no ball has an element of $G$ which swaps its two sides.
	This can be done by replacing each parting ball $B$ with the pair $\partial\nu(B)\setminus\partial B^n$.
	The induced graph, with a vertex for each 0--handle and an edge for each 1--handle, must be a tree since $B^n$ is contractible.
	If this tree is a single vertex, we're done.
	Otherwise, let $Z_2$ be a 0--handle corresponding to a valence-one vertex, let $Z_1$ be the 0--handle corresponding to the adjacent vertex, and let $B$ be the parting ball corresponding to the  edge between them.
	Clearly the stabilizer of $Z_2$ leaves $B$ invariant, so by Lemma~\ref{lem:cancel_parting}, we can contract the orbit of $Z_2$.
	
	The new linear parting will correspond to a smaller tree, so we can repeat this process inductively to remove all leaves, at which point the corresponding tree must be a single vertex, so the action is linear.
\end{proof}

We conclude this subsection by observing that actions on 4--dimensional 1--handlebodies need not be linearly parted, as mentioned above.

\begin{corollary}
\label{cor:not_parted_B4}
	There are $G$--actions on the 4--ball that are not linearly parted.
	Any equivariant handle-decomposition for such an action necessarily contains 2--handles.
\end{corollary}

\begin{proof}
	By the negative resolution of the Smith conjecture in dimension four~\cite{Gif_66_The-generalized-Smith-conjecture,Gor_74_On-the-higher-dimensional-Smith}, there exist non-linear actions on $B^4$.
	(Puncture a non-linear action on $S^4$ at a fixed point.)
	By Corollary~\ref{cor:parted_Bn}, any such action is not linearly parted.
	
	Suppose such an action has a handle-decomposition with no 2--handles.
	The union of the 0--handles and 1--handles is linearly parted, by definition, and must be a 4--ball.
	By the proof of Corollary~\ref{cor:parted_Bn}, all the 1--handles can be canceled.
	So, the action has an equivariant handle-decomposition with no 1--handles nor 2--handles.
	
	Capping-off the $S^3$ boundary,  turning the handle-decomposition upside down, and repeating the above argument removes all 3--handles.
	So, the action has an equivariant handle-decomposition with only 0--handles and 4--handles.
	Such a handle-decomposition must be a single 0--handle, hence the action is linear, a contradiction.
\end{proof}

\section{Equivariant Laudenbach-Po\'enaru}
\label{sec:ELP}

We are now ready to address our first main theorem, the proof of which depends on a handful of lemmata that are proved throughout the remainder of the section.

\begin{theorem}
\label{thm:laud_poen}
	Let $G$ be a finite group acting on $Y=\#^k(S^1\times S^2)$.
	\begin{enumerate}
		\item There exists a 4--dimensional 1--handlebody $X$ with $\partial X = Y$ such that the action of $G$ on $Y$ extends to a linearly parted action on $X$.
		\item If $X$ and $X'$ are two 4--dimensional 1--handlebodies with $\partial X = Y = \partial X'$ and both $X$ and $X'$ have linearly parted $G$--actions extending the $G$--action on $Y$, then $X$ and $X'$ are $G$--diffeomorphic rel-boundary.
	\end{enumerate}
\end{theorem}

We refer to the 4--dimensional 1--handlebodies featured in Theorem~\ref{thm:laud_poen} as \emph{equivariant fillings} of the $G$--manifold $Y$.
Recall that a collection $\Ss$ of disjointly embedded two-spheres in $Y$ is called a \emph{sphere-system} for $Y$ if $Y\setminus\nu(\Ss)$ is a collection of punctured 3--spheres.

\begin{proof}
	Part (1) follows from a straight-forward application of the Equivariant Sphere Theorem of Meeks and Yau~\cite{MeeSimYau_82_Embedded-minimal}, together with Lemma~\ref{lem:unique_handle_attachment}, which establishes that equivariant 3--handles and 4--handles are determined by their attaching spheres.
	Dunwoody gave a topological proof of the Equivariant Sphere Theorem~\cite{Dun_85_An-equivariant-sphere}, and they established the following generalization; see~\cite[Theorem~4.1]{Dun_85_An-equivariant-sphere}:
	Given a $G$--action on $Y$, there exists a $G$--invariant sphere-system $\Ss$ for $Y$.
	
	Given the invariant sphere-system $\Ss$, it is easy to build the desired equivariant filling $X=\natural^k(S^1\times B^3)$.
	To do this, attach $G$--equivariant, 4--dimensional 3--handles to $Y$ along each sphere in $\Ss$.
	Let $\Bb$ denote the cores of these 3--handles.
	Since $\Ss$ is a sphere-system, the result of attaching these 3--handles is a 4--manifold with some number of 3--sphere boundary components and a boundary component diffeomorphic to $Y$.
	The induced action on these spherical boundary components is linear, so we can attach $G$--equivariant 4--dimensional 4--handles to these $3$--spheres.
	The result is an (upside-down) 4--dimensional 1--handlebody $X$ with a $G$--action extending the $G$--action on $Y$ that is linearly parted by $\Bb$, as desired.
		
	For part (2), suppose $X$ and $X'$ are the two given equivariant fillings of $Y$.
	Let $\Bb$ and $\Bb'$ denote the linearly parting ball-systems in $X$ and $X'$, respectively, and let $\Ss = \partial\Bb$ and $\Ss' = \partial\Bb'$.
	If necessary, ensure that $G$ acts orientation-preservingly on $\Bb$ and $\Bb'$ by replacing each parting ball $B$ for which the induced action is orientation-reversing with the two balls of $\partial(\nu(B))\setminus Y$.
	This process preserves the fact that $\Bb$ and $\Bb'$ are linear parting systems, since the change adds a new $4$--ball to $X\setminus \Bb$, namely $\nu(B)$, on which the action of $G$ is linear, and the change removes a small portion of each $4$--ball adjacent to $B$, preserving the fact that that $4$--ball is linear.
	Thus we may assume that $G$ acts orientation-preservingly on each sphere of $\Ss$ and $\Ss'$.
	By Lemma \ref{lem:innermost}, there exists an invariant sphere-system $\Ss''$ for $Y$ which is disjoint from both $\Ss$ and $\Ss'$.
	By Lemma \ref{lem:new_ball-system}, there exist linearly parting ball-systems $\mathcal P$ and $\mathcal P'$ bounded by $\Ss''$ in $X$ and $X'$, respectively, that are disjoint from $\Bb$ and $\Bb'$, respectively.
	Finally, by Lemma \ref{lem:unique_handle_attachment}, since $X$ and $X'$ have isomorphic handle-decompositions, $X$ and $X'$ are equivariantly diffeomorphic by a diffeomorphism which is the identity on $Y$, as desired.
\end{proof}

\begin{remark}
	In the case of a trivial group action, the proof given above is, as far as we know, not currently present in the literature. The equivariant perspective required us to use many `redundant' spheres for a sphere-system, and required us to not rely significantly on techniques involving isotopies, so this approach results in a novel non-equivariant proof with those same technical features.
\end{remark}

We now prove the lemmata referenced in the proof of Theorem~\ref{thm:laud_poen}.
Let $n\in\{2,3\}$, and let $G$ act smoothly by isometries on a $S^n$.
In other words, $G$ acts linearly.
The group $\Isom(S^n,G)$ is the group of $G$--equivariant isometries of $S^n$, and the group $\Diff(S^n, G)$ is the group of $G$--equivariant diffeomorphisms of $S^n$.
For a linear $G$--action on a ball $B^{n+1}$, any $G$--equivariant isometry of $S^n$ is a linear map of $S^n$, hence extends smoothly and equivariantly over $B^{n+1}$ by coning.

\begin{lemma}
\label{lem:unique_handle_attachment}
	Let $X$ be a smooth $4$--dimensional $G$--manifold.
	Let $X' = X\cup \frak h$, where $\frak h$ is an equivariant $3$--handle or $4$--handle.
	The equivariant diffeomorphism type of $X'$ depends only on the equivariant isotopy type of the attaching sphere of $\frak h$.
\end{lemma}

\begin{proof}
	Proposition 2.1 of \cite{DinLee_09_Equivariant_Ricci_flow} (in particular, its Corollary, 2.6) shows that for any $G$--action on $S^2$, the natural inclusion $\Isom(S^2,G)\to \Diff(S^2,G)$ is a homotopy equivalence.
	In particular, any $G$--equivariant diffeomorphism of $S^2$ is equivariantly isotopic to a $G$--equivariant linear map of $S^2$, so any two attaching maps for a linear $3$--handle extend over the core $3$--ball, hence give equivariantly diffeomorphic results.

	The main theorem of \cite{MecSep_19_Isometry_groups} shows that for any $G$--action on $S^3$, the natural inclusion $\Isom(S^3,G)\to \Diff(S^3,G)$ is a $\pi_0$--isomorphism, i.e. any $G$--equivariant diffeomorphism of $S^3$ is, as above, isotopic to a $G$--equivariant linear map of $S^3$ and so any two attaching maps for a linear $4$--handle extend over that $4$--handle and hence give equivariantly diffeomorphic results.
\end{proof}

Note that Lemma~\ref{lem:unique_handle_attachment} establishes the special case of Theorem~\ref{thm:laud_poen}(2) in which $X$ is $B^4$ with a linear action.
See also \cite[Corollary~1.5]{ChoLi_24_Equivariant-3-manifolds-with} for the result that the natural inclusion $\Isom(S^3,G)\to \Diff(S^3,G)$ is a homotopy equivalence.

To prove Theorem~\ref{thm:equiv_patch}, we will also need the analogue of Lemma~\ref{lem:unique_handle_attachment} for equivariant $2$--handles, which we can also obtain from the results cited in Lemma~\ref{lem:unique_handle_attachment}.

\begin{lemma}
\label{lem:2_handle_attachment}
	Let $X$ be a smooth $4$--dimensional $G$--manifold.
	Let $X' = X\cup \frak h$, where $\frak h$ is an equivariant $2$--handle.
	The equivariant diffeomorphism type of $X'$ depends only on the equivariant isotopy type of the attaching map $f\colon \partial D^{2}\times D^2\to \partial X$.
\end{lemma}

\begin{proof}
	For $i=1,2$, let $f_i\colon\partial D^2\times D^2\to \partial X$ be the attaching map of an equivariant 2--handle $\frak h_i$.
	Assume $f_1$ and $f_2$ are $G$--equivariantly isotopic.
	By the Equivariant Isotopy Extension Theorem (see~\cite[Remark, p. 312]{Bre_72_Introduction-to-compact-transformation-groups}), there is a $G$--equivariant diffeomorphism $\phi\colon\partial X\to\partial X$ such that $f_2 = \phi\circ f_1$.
	We can extend $\phi$ from $\partial D^2\times D^2$ to all of $\partial \frak h_1\cong S^3$ by the~\ref{ELT} and~\cite[Proposition~2.1]{DinLee_09_Equivariant_Ricci_flow}.
	Then, we can extend $\phi$ across $\frak h_1$ by the main theorem of~\cite{MecSep_19_Isometry-groups-and-mapping}.
	Since $\phi$ is $G$--equivariantly isotopic to the identity, it extends across $X$ to give a $G$--equivariant diffeomorphism $X\cup\frak h_1\to X\cup\frak h_2$, as desired.
\end{proof}

Central to our approach is the ability to make invariant sphere-systems transverse equivariantly, which we now establish.

\begin{lemma}
\label{lem:3d_trans}
	Let $Y$ be a closed, 3--dimensional $G$--manifold.
	Let $N$ and $M$ be 2--dimensional invariant submanifolds of $Y$ such that the induced action of $G$ on $N$ is orientation-preserving.
	Then $N$ and $M$ can be made transverse by an arbitrarily small equivariant isotopy. 
\end{lemma}

\begin{proof}
	We first characterize the singular set $\Sing(Y)$ of the action of $G$ on $Y$ and its intersection with $N$.
	Let $p$ be a point of $\Sing(Y)\cap N$ with stabilizer $G_p$.
	Let $\nu(p)$ be an invariant neighborhood of $p$ intersecting $N$ in a disk, which is equivariantly diffeomorphic to the unit ball $B^3\subset\R^3$ via a $G_p$--representation $\rho\colon G_p\to O(3)$.
	Under this identification, $N$ corresponds to an invariant 2--dimensional subspace in $B^3$, so $\rho$ is reducible and splits as the direct sum $\rho_2\oplus\rho_1$ of a 2--dimensional real representation and a 1--dimensional real representation.
	If $\rho_2(G_p)$ contains a reflection, then the restriction of $G$ to $N$ is not orientation-preserving, a contradiction; so $\rho_2(G_p)$ is cyclic.
	In this case, if $\rho_1(G_p)$ is non-trivial, then $G_p$ (hence $G$) is not orientation-preserving on $Y$ at $p$, another contradiction.
	It follows that $\rho_1(G_p)$ is trivial and thus $\rho(G_p)$ is cyclic.

	Geometrically, this means that $\Sing(Y)\cap \nu(p)$ is an arc transverse to $N$ at $p$.
	Suppose that $N$ intersects $M$ at a point $p$ of $\Sing(Y)$.
	We have two cases for such a point $p$: first, if the action of $G_p$ on $M$ is orientation-reversing, then the arc $\Sing(Y)\cap \nu(p)$ must be embedded in $M$ and transverse to $N$, and therefore $N$ is transverse to $M$ at $p$.
	
	In the other case, the action of $G_p$ on $M$ is orientation-preserving, so by the same analysis as above $\Sing(Y)$ is transverse to $M$ at $p$ in addition to $N$.
	Passing to the orbifold $Y/G$, we see that, at the quotient $p^*$ of $p$, both $N/G$ and $M/G$ are transverse to the singular set $\Sing(Y)/G$, which is a $1$--manifold near $p^*$.
	We can now isotope $N/G$ to be disjoint from $M/G$ on $\Sing(Y)/G$ via an arbitrarily small isotopy that is a vertical translation of $N/G$ along the arc of $\Sing(Y)/G$ passing through $p^*$.
	This isotopy lifts to an arbitrarily small equivariant isotopy of $N$ in $Y$ by~\cite[Corollary~2.4]{Sch_80_Lifting-smooth-homotopies}.
	
	Finally, if $N$ and $M$ intersect away from the singular set of the $G$--action on $N$, we can apply an arbitrarily small isotopy that is supported away from $\Sing(Y)/G$ to $N/G$ so that it meets $M/G$ transversally.
	Since this isotopy is the identity near $\Sing(Y)/G$, it lifts to an equivariant isotopy in $Y$, after which $N$ and $M$ intersect transversally as desired.
\end{proof}

Next, we leverage 3--dimensional equivariant transversality to prove a lemma regarding equivariant sphere-systems.

\begin{lemma}
\label{lem:innermost}
	Let $G$ act on $Y=\#^k(S^1\times S^2)$, and suppose $\Ss$ and $\Ss'$ are two $G$--invariant sphere-systems for $Y$ so that $G$ acts orientation-preservingly on each sphere of $\Ss$ and $\Ss'$.
	There exits a $G$--invariant sphere-system $\Ss''$ that is disjoint from both $\Ss$ and $\Ss'$ so that $G$ acts orientation-preservingly on each component of $\Ss''$.
\end{lemma}

\begin{proof}
	Because $G$ acts orientation-preservingly on $\Ss$, we can apply Lemma~\ref{lem:3d_trans} to equivariantly isotope $\Ss'$ to be transverse to $\Ss$; let $\Cc = \Ss\cap\Ss'$ be the resulting disjoint union of simple closed curves.
	Choose a curve $\omega\subset\Cc$ such that $\omega = \partial D$ for some disk $D\subset\Ss$ with the property that $\Int(D)\cap\Ss' = \varnothing$.
	Let $\Dd$ denote the orbit of $D$ under the action of $G$, so $\Dd$ is an invariant disjoint union of innermost disks in $\Ss$.
	Using an equivariant tubular neighborhood of $\Dd$, surger $\Ss'$ along $\Dd$ to get a new sphere-system $\Ss''$.
	
	To check that $\Ss''$ is a sphere-system, note that $E'=Y\setminus\nu(\Ss')$ is a disjoint union of punctured 3--spheres, and $\Dd$ is a disjoint union of neatly embedded disks in $E'$.
	To get $E'' = Y\setminus\nu(\Ss'')$ from $E'$, we remove $\nu(\Dd)$ from some punctured 3--spheres in $E'$, then glue $\nu(\Dd)$ to (potentially different) punctured 3--spheres in $E'$ along the boundaries of disks in $\Dd$.
	See Figure~\ref{fig:surgery_lemma}(\textsc{a}):
	Removing $\nu(\Dd)$ from $E'$ results in a disjoint union of punctured 3--spheres, since a neatly embedded disk in a 3--ball always cuts out two 3--balls.
	Gluing $\nu(\Dd)$ back in has the effect of attaching an equivariant collection of $2$--handles to $E'$.
	Given that a collection of punctured $3$--balls is equivalently a 3--dimensional 2--handlebody, this results in a collection of punctured 3--spheres again.
	Hence $E''$ is a collection of punctured $3$--spheres, and thus $\Ss''$ is a sphere-system for $Y$.
	Note that $G$ acts orientation-preservingly on the spheres of $\Ss''$, since $G$ acts orientation-preservingly on the spheres of $\Ss'$ and the invariant disks of $\Dd\subset\Ss$.
	
\begin{figure}[ht!]
	\centering
	\includegraphics[width = .9\linewidth]{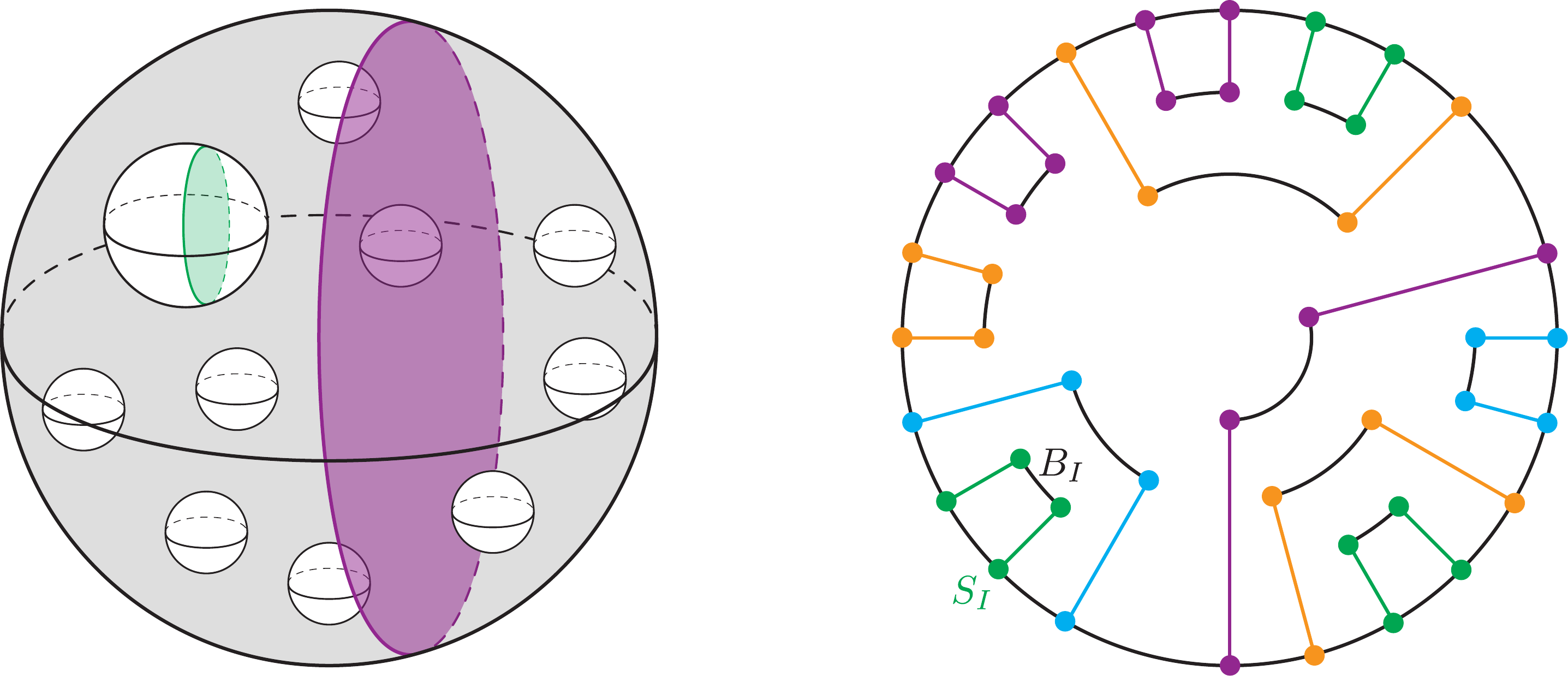}
    \caption{Left: Removing the purple disk on the right from this punctured $3$--ball splits it into two punctured $3$--balls, and attaching the green disk on the left splits the puncture in two, so that the left-hand punctured $3$--ball turns into another punctured $3$--ball with one more puncture, as in Lemma~\ref{lem:innermost}.
    Right: A schematic for the tube-and-cap approach in Lemma~\ref{lem:new_ball-system}, with an initial choice of innermost sphere $S_I$ and capping 3--ball $B_I$ indicated.}
    \label{fig:surgery_lemma}
\end{figure}

	Now, apply Lemma~\ref{lem:3d_trans} again to obtain a sufficiently small equivariant isotopy to make $\Ss''$ transverse to the original sphere-system $\Ss'$ while preserving the number of curves of intersection with $\Ss$.
	Repeatedly applying the surgery argument to $\Cc' = \Ss'\cap\Ss''$, we reduce the cardinality of $\Cc'$.
	Since $|\Cc'|$ is finite, this process terminates in a sphere-system (still denoted $\Ss''$) that is disjoint from $\Ss'$ and that still intersects $\Ss$ in a set of curves equivalent to $\Cc\setminus(G\cdot \omega)$, as desired.

	Note that $|\Ss''\cap\Ss| < |\Ss'\cap\Ss|$.
	Iterating this entire process results in a sequence of sphere-systems (still denoted $\Ss''$), each of which is disjoint from $\Ss'$ and intersects $\Ss$ in fewer curves.
	As above, since $|\Cc|$ was finite, this sequence terminates in a sphere-system $\Ss''$ that is disjoint from $\Ss$ and $\Ss'$, and the $G$--action on $\Ss''$ is orientation-preserving, as desired.
\end{proof}

We remark that in the above proof, instead of repeating the surgery argument to get $\Ss''$ disjoint from $\Ss'$, we could use invariant tubular neighborhoods to push $\Ss''$ off of $\Ss'$.

The next lemma concerns extending sphere-systems to ball-systems and shows that two disjoint sphere-systems for $Y$ detemine the same linearly parted handlebody upon equivariant handle-attachment.

\begin{lemma}
\label{lem:new_ball-system}
	Let $G$ act on $X=\natural^k(S^1\times B^3)$, and let $\Bb$ be a linearly parting ball-system for $X$.
	Let $\Ss''$ be a $G$--invariant sphere-system for $\partial X$ that is disjoint from $\partial\Bb$.
	There is a linearly parting ball-system $\Bb''$ for $X$ such that $\Bb''$ is disjoint from $\Bb$ and $\partial\Bb'' = \Ss''$.
\end{lemma}

\begin{proof}
	Note that $X\setminus\nu(\Bb)$ is a disjoint union of 4--balls on which $G$ acts linearly, and $\Ss''$ is a $G$--invariant disjoint union of two-spheres in the disjoint union of 3--spheres $\partial(X\setminus\nu(\Bb))$.
	Let $Z$ be one of these 4--balls, let $\Ss''_Z$ denote those spheres of $\Ss''$ lying in $\partial Z$, and let $G_Z$ be the stabilizer of $Z$.
	There are finitely many spheres in $\Ss''_Z$, so there is an innermost sphere $S_I$ that bounds a 3--ball $B_I$ disjoint from the other spheres of $\Ss''_Z$.
	Note first that since $G$ is orientation-preserving, any element of $G$ that takes $S_I$ to itself must take $B_I$ to itself, rather than the other ball bounded by $S_I$. 
	Thus $G_Z\cdot B_I$ is an invariant, disjoint collection of $3$--balls bounded by $G_Z\cdot S_I$, and $G_Z\cdot B_I$ is disjoint from all other spheres of $\Ss''_Z$.
	Since the $G_Z$--action is a cone, we can cap off $G_Z\cdot S_I$ with $G_Z\cdot B_I$ and tube the spheres of $\Ss''_Z\setminus (G_Z\cdot S_I)$ down the depth function.
	See Figure~\ref{fig:surgery_lemma}(\textsc{b}).
	Repeat this process finitely many times for all the $G_Z$--orbits in $\Ss''_Z$ to obtain an invariant, disjoint collection $\Bb''_Z$ of $3$--balls in $Z$ bounded by $\Ss''_Z$.

	Each component of $Z\setminus \Bb''_Z$ is a linear $4$--ball: This is clear for the component containing the center of $Z$, since it has an equivariant Morse function with a single critical point\footnote{Also, this component is a starlike sub-ball of $Z$, as in the end of Section~\ref{sec:background}.}.
	For any other component $C$, if we let $B\subseteq\partial C$ be the ball of $\Bb''_Z$ of lowest height in $\partial C$, then $C$ is the flow of $B$ up the cone height function to $\partial Z$ minus carvings from the boundary corresponding to spheres contained within this flow on $\partial Z$.
	The action on $C$ is linear since it is diffeomorphic to carvings from $B\times[0,1]$, and the action on $B$ is linear since $B$ has dimension three; see Remark~\ref{rmk:automatic_linear}.

	Taking the $G$--orbit of $\Bb''_Z$ provides balls for all spheres in $\Ss''$ contained in the orbit of $Z$.
	Repeat this process across all orbits of balls in $X\setminus\nu(\Bb)$ to obtain $\Bb''$, an invariant collection of disjoint $3$--balls bounded by $\Ss''$.
	It now suffices to show that $\Bb''$ is a linearly parting ball-system.
	
	Every component of $X\setminus\nu(\Bb'')$ is the union of linear 4--balls along the 3--balls of $\Bb$; hence, every such component is a linearly parted 4--dimensional 1--handlebody.
	Suppose some component $W$ is not a 4--ball.
	Then, there is an essential (and not boundary-parallel) sphere in $\partial W$.
	But $\partial W$ is a subset of $\partial X\setminus\nu(\Ss'')$, which is irreducible since $\Ss''$ is a sphere-system.
	Thus each component of $X\setminus\nu(\Bb'')$ is a linearly parted $4$--ball, so by Corollary~\ref{cor:parted_Bn}, each such component is a linear $4$--ball, and hence $\Bb''$ is a linearly parting ball-system for $X$ bounded by $\Ss''$.
\end{proof}

We conclude this section by remarking on some equivalent formulations of theorems such as Theorem~\ref{thm:laud_poen}(2) and the classical Laudenbach-Po\'enaru theorem that may at first seem non-obvious.
Our recalling of the classical Laudenbach-Po\'enaru theorem in the abstract and introduction takes the form of Proposition~\ref{prop:equivalent_statements}(1), while our Theorem~\ref{thm:laud_poen}(2) takes the form of Proposition~\ref{prop:equivalent_statements}(3).

\begin{proposition}
\label{prop:equivalent_statements}
	Let $X$ and $X'$ be diffeomorphic manifolds.
	The following are equivalent:
	\begin{enumerate}
		\item For every $\gamma\in \Diff(\partial X)$, there exists $\Gamma\in \Diff(X)$ with $\Gamma|_{\partial X} = \gamma$, i.e. $\Gamma$ extends $\gamma$.
		\item For every $f\in\Diff(\partial X,\partial X')$, there exists an $F\in\Diff(X,X')$ such that $F\vert_{\partial X} = f$, i.e. $F$ extends $f$.
		\item If $\partial X = \partial X'$, then there exists $\Phi\in\Diff(X,X')$ such that $\Phi\vert_{\partial X} = \Id_{\partial X}$, i.e. $X$ and $X'$ are diffeomorphic \emph{rel-boundary}.
	\end{enumerate}
	Moreover, the above statements remain equivalent in the equivariant setting in which each manifold is equipped with a $G$--action for some finite group $G$ and each diffeomorphism is $G$--equivariant.
\end{proposition}

\begin{proof}
	The implications $(2)\Rightarrow(1)$ and $(2)\Rightarrow(3)$ are immediate.
	
	For the implication $(1)\Rightarrow(2)$, let $\Phi\in\Diff(X,X')$ with $\Phi\vert_{\partial X} = \varphi$, and let $f\in\Diff(\partial X, \partial X')$ be given.
	Let $\gamma = f^{-1}\circ\varphi \in \Diff(\partial X)$.
	By (1), there exists $\Gamma\in\Diff(X)$ such that $\Gamma\vert_{\partial X} = \gamma$.
	Let $F = \Phi\circ\Gamma^{-1}$, so $F\in\Diff(X, X')$ and $F\vert_{\partial X} = \phi\circ\gamma^{-1} = f$, as desired.
	
	For the implication $(3)\Rightarrow(2)$, let $f\in\Diff(\partial X, \partial X')$ be given.
	Using $f$, we think of $\partial X'$ as equal to $\partial X$.
	Formally, let $X'' = \partial X\cup_{f^{-1}}X'$, and note that $\partial X = \partial X''$ by construction.
	By (3), there exists $\Phi\in\Diff(X, X'')$ such $\Phi\vert_{\partial X} = \Id_{\partial X}$.
	Letting $F=\Phi$, we have $F\in\Diff(X,X')$ with $F\vert_{\partial X}=f$, as desired.
	
	All of this reasoning carries over to the equivariant setting.	
\end{proof}

\section{Equivariant disk-tangles in equivariant $1$--handlebodies}
\label{sec:ELP_surfaces}

Our second main theorem, Theorem~\ref{thm:equiv_patch}, is an adaptation of Theorem~\ref{thm:laud_poen} to the setting of pairs $(X,\Dd)$, where $X\cong\natural^k(S^1\times B^3)$ and $\Dd\subset X$ is a boundary-parallel disk-tangle.
This gives an equivariant version of~\cite[Lemma~8]{MeiZup_18_Bridge-trisections}, which built on earlier work of Livingston~\cite{Liv_82_Surfaces-bounding-the-unlink}.
Before proving Theorem~\ref{thm:equiv_patch}, we introduce the relevant definitions and prove the required lemmata.
The following definitions are motivated by and developed in detail in~\cite{MeiSco_tri}.

\begin{definition}
	Let $X$ be a 4--dimensional 1--handlebody.
	A \emph{$p$--patch disk-tangle} $\Dd$ is a disjoint collection of neatly embedded disks in $X$.
	A disk-tangle $\Dd$ is \emph{boundary-parallel} if $\Dd$ can be isotoped relative to $\partial\Dd$ to lie in $\partial X$.
	Equivalently, if $\Dd = \cup_{i=1}^p D_i$, there are pairwise disjoint $3$--balls $\Delta = \cup_{i=1}^p\delta_i$ such that $\partial\delta_i = F_i\cup D_i$, where $F_i$ is a spanning disk for the component $\partial D_i$ of the unlink $L=\partial \Dd$ in $\partial X$.

	The balls of $\Delta$ are collectively call \emph{bridge-balls} for the disk-tangle $\Dd$, and the disks $\Ff = \cup_{i=1}^pF_i$ are called \emph{Seifert disks} for $L$.
	
\end{definition}

\begin{definition}
	Let $G$ act on a $4$--dimensional $1$--handlebody $X$.
	An invariant boundary-parallel disk-tangle $\Dd = \cup_i D_i$ in $X$ is \emph{equivariantly boundary-parallel} if there is a collection $\Delta = \cup_i\delta_i$ of bridge-balls that is $G$--invariant and \emph{almost-disjoint}: for each $i,j$, either $\delta_i\cap\delta_j = \varnothing$ or, if $\delta_i, \delta_j$ are bridge-balls for the same disk $D_k$, $\delta_i\cap\delta_j = D_k$.

	A bridge-ball $\delta$ (respectively, a Seifert disk $F$) will be called \emph{almost-equivariant} if its orbit is an almost-disjoint collection of bridge-balls (respectively, Seifert disks).
	
\end{definition}

In dimension three, it is possible to use the~\ref{ELT} to prove that an invariant, boundary-parallel tangle in a $1$--handlebody is equivariantly boundary-parallel; see ~\cite[Subsection~3.2]{MeiSco_tri} for a complete discussion.
Owing to the lack of a 4--dimensional analog of the Equivariant Loop Theorem, we can but pose the following question regarding dimension four.

\begin{question}
	\label{ques:disk-tangle-parallel}
	Let $G$ act on a 4--dimensional 1--handlebody $X$, and let $\Dd$ be a $G$--invariant, boundary-parallel disk-tangle in $X$.
	Is $\Dd$ equivariantly boundary-parallel?
\end{question}

We can now introduce the relevant class of actions in this setting.

\begin{definition}
\label{def:linearly_parted_pair}
	Let $G$ act on  a $4$--dimensional $1$--handlebody $X$.
	Let $\Dd$ be an equivariantly boundary-parallel disk-tangle in $X$.
	We say that the pair $(X,\Dd)$ is \emph{linearly parted as a pair} if there exists a linearly parting system $\Bb$ for $X$ and a collection of equivariant bridge-balls $\Delta$ for $\Dd$ such that $\Bb\cap\Delta = \varnothing$ and each $4$--ball component of $X\setminus\Bb$ contains at most one disk of $\Dd$.
\end{definition}

Again, owing to the lack of analogues to the $3$--dimensional techniques, we can but pose the following question regarding whether equivariantly boundary-parallel disk-tangles are always linearly parted as pairs.
This question is more technical than the previous one but perhaps also more tractable.

\begin{question}
	Let $G$ act on a 4--dimensional 1--handlebody $X$, and suppose that $X$ is linearly parted. 
	Let $\Dd$ be an equivariantly boundary-parallel disk-tangle in $X$.
	Is $(X,\Dd)$ linearly parted as a pair?
\end{question}

We conclude this introduction to equivariant boundary-parallel disks with an interesting corollary to Theorem~\ref{thm:laud_poen} that uses Proposition~\ref{prop:EST_pairs} and Theorem~\ref{thm:equiv_patch}(1) (both of which are essentially corollaries of the Equivariant Sphere Theorem). 
This corollary offers solutions to the above questions in some special cases.
In particular, we show that under certain hypotheses, whether an invariant surface in a handlebody is linearly parted as a pair is determined by its boundary.

\begin{corollary}
\label{coro:must-be-disk-tangle}
	Let $G$ act on a $4$--dimensional $1$--handlebody $X$ so that the action is linearly parted.
	Let $\Dd\subset X$ be a neatly embedded, invariant surface and suppose that each component of $\Dd$ is fixed point-wise by a non-trivial element of $G$.
	Then $\Dd$ is an equivariantly boundary-parallel disk-tangle and the action on $(X,\Dd)$ is linearly parted as a pair if and only if $\partial \Dd\subseteq \partial X$ is an unlink.
\end{corollary}

\begin{proof}
	We begin with the forward direction: letting $\Delta$ be a bridge-ball system for $\Dd$, $\Delta\cap \partial X$ is an almost-disjoint collection of Seifert disks for $\partial \Dd$ which demonstrates that $\partial \Dd$ is an unlink.

	Now for the reverse direction.
	By Proposition~\ref{prop:EST_pairs}, there exists a $G$--invariant sphere-system $\Ss$ for $Y = \partial X$ and a $G$--invariant, almost-disjoint collection $\Ff$ of Seifert disks for $L = \partial\Dd$ such that $\Ss\cap\Ff = \varnothing$ and $\Ss$ separates the components of $L$.
	By Theorem~\ref{thm:equiv_patch}(1), there exists a $G$--equivariant filling $(X',\Dd')$ of $(Y,L)$ along $\Ss$ and $\Ff$ that is linearly parted as a pair.
	By Theorem~\ref{thm:laud_poen}, there is a $G$--equivariant diffeomorphism rel-boundary $\psi\colon X'\to X$.
	
	Let $g\in G$ with prime order. Let the $2$--dimensional part of the fixed-point set of $g$ in $X$ be $\Fix_{X}(g)$ and similar for $\Fix_{X'}(g)$ and note that each set consists of a disjoint union of neatly-embedded surfaces since $g$ acts orientation-preservingly on $X$ and $X'$.
	Since $\psi$ is $G$--equivariant, it is $g$--equivariant for all $g\in G$, so $\psi(\Fix_{X'}(g)) = \Fix_{X}(g)$.

	Let $\Dd''\subseteq X$ be $\psi(\Dd')$.
	Note that since $\psi$ is the identity on $\partial X$, $\partial D' = \partial \psi(D')$.
	Note also that by pushing forward the linear parting system and bridge ball systems from $X'$, the action of $G$ on $(X,\Dd'')$ is linearly parted as a pair.
	
	Let $D\in \Dd$. We claim that for the unique disk $\psi(D')$ of $\Dd''$ with $\partial \psi(D') = \partial D' = \partial D$, $D = \psi(D')$.
	Let $g\in G$ fix $\partial D$, and by taking a power of $g$ if necessary, assume that $g$ is prime order.
	The diffeomorphism $\psi$ maps $\Fix_{X'}(g)$ to $\Fix_{X}(g)$.
	There is a unique connected component of $\Fix_{X'}(g)$ which contains $\partial D'$, namely $D'$, and $\psi$ maps $\partial D'$ to $\partial D$ by assumption, hence $\psi$ maps the unique connected component of $\Fix_{X'}(g)$ containing $\partial D'$ to the unique connected component of $\Fix_{X}(g)$ containing $\partial D' = \partial D$, i.e. $\psi$ maps $D'$ onto $D$.
	Thus $\Dd = \Dd''$, and so $(X,\Dd)$ is linearly parted as a pair.
\end{proof}

Let $G$ act on $Y = \#^k(S^1\times S^2)$, and let $L\subset Y$ be a $G$--invariant unlink. 
A \emph{$G$--equivariant filling} of the pair $(Y,L)$ is a pair $(X,\Dd)$, where $X$ is an equivariant filling of $Y$ (as in the previous section), $\Dd$ is an equivariantly boundary-parallel disk-tangle in $X$ with $\partial\Dd = L$, and the pair $(X, \Dd)$ is linearly parted as a pair.

We require the following adaptation of the Equivariant Sphere Theorem to the present setting.

\begin{proposition}
\label{prop:EST_pairs}
	Let $G$ act on $Y=\#^k(S^1\times S^2)$, and let $L\subset Y$ be a $G$--invariant unlink.
	There exists a $G$--invariant sphere-system $\Ss$ for $Y$ and a $G$--invariant, almost-disjoint collection $\Ff$ of Seifert disks for $L$ such that $\Ss\cap\Ff = \varnothing$ and each component of $Y\setminus\nu(\Ss)$ contains at most one component of $L$.
\end{proposition}

\begin{proof}
	Let $\nu(L)$ be an open $G$--invariant tubular neighborhood of $L$ in $Y$, and let $M = Y\setminus \nu(L)$.
	By~\cite[Theorem~4.1]{Dun_85_An-equivariant-sphere}, $M$ admits an invariant collection of pairwise disjoint two-spheres such that each component of $M\setminus\nu(\Ss)$ is irreducible.
	It follows that $\Ss$ is an invariant sphere-system for $Y$ and that each component of $Y\setminus\nu(\Ss)$ contains at most one component of $L$.
	
	Continuing, now let $M' = Y\setminus\nu(L\cup\Ss)$.
	Let $\Sigma$ a torus component of $\partial M'$, and let $C\subset\Sigma$ be a longitude for the component $K\subset L$ corresponding to $\Sigma$.
	We will show that after a small equivariant isotopy of $C$, $M'$ and $C$ satisfy the hypotheses of the~\ref{ELT}.
	
	Since $G$ acts on $\nu(K)$ orientation-preservingly, the exceptional set of the action of $G$ on $\nu(K)$ is a graph, so the exceptional set $E$ of the action of $G$ on $\Sigma$ is a finite collection of points.
	We can equivariantly isotope $C$ to a parallel curve that is disjoint from $E$, establishing transversality as required by part (3) of the Equivariant Loop Theorem.
	Since $C$ is a longitude for the unknotted and unlinked component $K\subset L$, we have that $C$ is null-homotopic in $M'$.
	It follows that $M'$ and $C$ satisfy the hypotheses of the Equivariant Loop Theorem.

	Let $D$ denote the equivariant disk in $M'$ with $\partial D=C$ guaranteed by the Equivariant Loop Theorem.
	Adding a collar to $D$ in $\nu(L)$, we obtain an almost-equivariant Seifert disk $F_K$ for the component $K$ of $L$.
	Let $\Ff'_K$ denote the orbit of $F_K$, an almost-disjoint collection of Seifert disks for the orbit of $K$.
	
	Now, let $M''=Y\setminus\nu(L\cup\Ss\cup \Ff'_K)$.
	Choose a torus component of $\partial M''$, and repeat the above argument to get an invariant, almost-disjoint collection of Seifert disks for the orbit of the component of $L$ corresponding to that torus.
	The new Seifert disks will be disjoint from $\Ff'_K$, since $\nu(\Ff'_K)$ was excluded from $M''$.
	Repeating this approach, we obtain a $G$--invariant union $\Ff$ of almost-disjoint Seifert disks for $L$ that is disjoint from the invariant sphere-system $\Ss$ and has the property that there is at most one component of $L$ in each component of $Y\setminus\nu(\Ss)$, as desired.
\end{proof}

We next require a version of Lemma~\ref{lem:unique_handle_attachment} adapted to the present setting.
This will be Corollary~\ref{cor:disk_fillings_diffeomorphic}, which follows from our next lemma.

\begin{lemma}
\label{lem:disks_in_ball_diffeomorphic}
	Let $G$ act linearly on $B^4$, and let $U$ be an unknot in $\partial B^4 = S^3$. Let $D$ and $D'$ be two invariant, neatly embedded, equivariantly boundary-parallel disks in $B^4$ with $\partial D = \partial D' = U$.
	There exists an equivariant diffeomorphism $\phi$ of $B^4$ that is the identity on $\partial B^4$ with $D' = \phi(D)$.
\end{lemma}

\begin{proof}
	Since $D$ is equivariantly boundary-parallel, $X = B^4 \setminus \nu(D)$ is diffeomorphic to $S^1\times B^3$.
	The same is true of $X' = B^4\setminus\nu(D')$.
	The induced $G$--actions on $X$ and $X'$ are linearly parted by ball-systems $\Bb$ and $\Bb'$, respectively, that come from the almost-equivariant bridge-balls for $D$ and $D'$.
	The boundary of $X$ is  $Y = (S^3\setminus\nu(U))\cup V$, the result of equivariant $0$--framed Dehn surgery on the invariant unknot $\partial D = U$, with Dehn filling solid torus $V$.
	Symmetrically, we see that $\partial X'$ has the same form, $Y' = (S^3\setminus\nu(U))\cup V'$, with Dehn filling solid torus $V'$.
	Since the framings for these surgeries are the same, as in the proof of Lemma~\ref{lem:2_handle_attachment}, this gives an equivariant diffeomorphism $f\colon Y\to Y'$.
	By Theorem~\ref{thm:laud_poen}(2) and the equivalence of statements (2) and (3) in Proposition~\ref{prop:equivalent_statements}, there exists an equivariant diffeomorphism $F\colon X\to X'$ such that $F\vert_{\partial X} = f$.

	Since $B^4$ is obtained by attaching $G$--equivariant 2--handles $\frak h$ and $\frak h'$ to $X$ or $X'$ along $V$ and $V'$, respectively, by Lemma~\ref{lem:2_handle_attachment}, the $G$--equivariant diffeomorphism $F\colon X\to X'$ extends to a $G$--equivariant diffeomorphism of $B^4$ taking the cocore $\frak h$ to the cocore of $\frak h'$.
	These cocores are, respectively, $D$ and $D'$, completing the proof.
\end{proof}

We remark that this lemma does not establish that the disks are equivariantly isotopic.
Establishing this stronger property seems to be a challenging problem, since the usual techniques in the non-equivariant setting require isotopies that cannot be performed equivariantly.
In light of the above lemma, we have the following corollary in analogy with Lemma~\ref{lem:unique_handle_attachment}.

\begin{corollary}
\label{cor:disk_fillings_diffeomorphic}
	Let $G$ act linearly on $(S^3, U)$.
	Let $(B, D)$ and $(B', D')$ be two fillings of $(S^3, U)$.
	Then, $(B, D)$ and $(B'D')$ are diffeomorphic rel-boundary as pairs.
\end{corollary}

\begin{proof}
	By Lemma~\ref{lem:unique_handle_attachment}, there is an equivariant diffeomorphism $\psi\colon B\to B'$ restricting to the identity on $S^3$.
	Then, $D'$ and $\psi(D)$ are invariant, boundary-parallel, neatly embedded disks in $B'$.
	By Lemma~\ref{lem:disks_in_ball_diffeomorphic}, there is a diffeomorphism $\Phi\colon (B',\psi(D))\to (B',D')$ restricting to the identity on the boundary.
	The desired diffeomorphism is $\Phi\circ\psi$.
\end{proof}

We are now ready to prove the main result of this section.

\begin{theorem}
\label{thm:equiv_patch}
	Let $G$ act on $Y=\#^k(S^1\times S^2)$, and let $L\subset Y$ be a $G$--invariant unlink.
	\begin{enumerate}
		\item There exists a $G$--equivariant filling $(X,\Dd)$ of $(Y,L)$.
		\item If $(X,\Dd)$ and $(X',\Dd')$ are two $G$--equivariant fillings of $(Y,L)$, then $(X,\Dd)$ and $(X',\Dd')$ are $G$--diffeomorphic rel-boundary.
	\end{enumerate}
\end{theorem}

\begin{proof}
	Part (1) follows easily from Proposition~\ref{prop:EST_pairs} as follows.
	Attach equivariant 3--handles to $Y$ along the invariant sphere-system $\Ss$ provided by Proposition~\ref{prop:EST_pairs}, and cap off the resulting 3--manifold with equivariant 4--handles to obtain $X$, just as in the proof of Theorem~\ref{thm:laud_poen}(1).
	The $G$--invariant almost-disjoint Seifert disks $\Ff$ for $L$ lie in the boundary of the 4--handles.
	Since the action on $X$ is linearly parted by the cores of the 3--handles, the induced action on each 4--ball is simply the cone on the action on its boundary 3--sphere.
	There is at most one component of $L$ in each such 3--sphere.
	For each component $K$ of $L$, take the disk $D_K$ that $K$ bounds in its $4$--ball to be the cone on $K$, and note that the cone on the almost-disjoint collection of Seifert disks $\Ff_K$ for $K$ gives an invariant, almost-disjoint bridge-ball-system for $D_K$, showing that $D_K$ is boundary-parallel in its $4$--ball.
	Repeating this construction over all components of $L$ gives the required disk-tangle $\Dd$.

	For part (2), we closely follow the proof of Theorem~\ref{thm:laud_poen}(2), indicating how to adapt the key lemmata of Section~\ref{sec:ELP} to the present setting of pairs.
	In particular, we will carefully track the invariant, almost-disjoint Seifert disks, along with the invariant sphere-systems.
	Let $(X, \Dd)$ and $(X', \Dd')$ be two equivariant fillings of $(Y, L)$.
	Let $\Bb\cup\Delta$ and $\Bb'\cup\Delta'$ denote the linearly parting $3$--balls and invariant collections of bridge-balls for the pairs $(X, \Dd)$ and $(X', \Dd')$, respectively.
	Let $\Ss = \partial \Bb$ and $\Ss' = \partial \Bb'$ denote the associated invariant sphere-systems, and let $\Ff$ and $\Ff'$ denote the invariant, almost-disjoint Seifert disks corresponding to $\Delta$ and $\Delta'$.
	
	By the same technique as in the proof of Theorem~\ref{thm:laud_poen}(2), we may assume the $G$--action on each ball of $\Bb$, and hence each sphere of $\Ss$, is orientation-preserving.
	By working in the complement of the disks $\Dd$, we can assume the $G$--action on each bridge-ball of $\Delta$, and hence each Seifert disk of $\Ff$, is orientation preserving.
	Do the same for the corresponding objects in $X'$.
	In Lemma~\ref{lem:innermost}, we produced a sphere-system $\Ss''$ disjoint from $\Ss$ and $\Ss'$, on each component of which the $G$--action is orientation-preserving.
	This involved equivariantly isotoping $\Ss$ and $\Ss'$ to be transverse using Lemma~\ref{lem:3d_trans} and surgering $\Ss'$ against $\Ss$ and itself repeatedly.
	Here, we do the same, but instead we surger $\Ss'\cup \Ff'$ against $\Ss\cup\Ff$ and itself, yielding an equivariant sphere-system $\Ss''$ and an almost-equivariant Seifert disk-system $\Ff''$ for $L$ that are disjoint from one another and from $\Ss$ and $\Ss'$.
	
	There are a few differences between this setting and that of Lemma~\ref{lem:innermost}, which we now outline.
	First, working in a tubular neighborhood $\nu(L)$ of $L$, we can equivariantly isotope $\Ff'$ relative to $L$ to be almost-disjoint to $\Ff$ in $\nu(L)$.
	Next, working away from $\nu(L)$, we can equivariantly isotope $\Ss'\cup \Ff'$ to be transverse to $\Ss\cup\Ff$ so that the equivariant isotopies are isotopies relative to $\partial(\nu(L))$ for $\Ff$ and $\Ff'$.
	This assures that $(\Ss\cup\Ff)\cap(\Ss'\cup\Ff')\setminus L$ is a invariant collection of simple closed curves $\Cc$, as in Lemma~\ref{lem:innermost}.
	
	In this new setting, the innermost disk $D$ with $\partial D = \omega$ may lie in either $\Ss$ or $\Ff$ and $\omega$ may lie in either $\Ss'$ or $\Ff'$.
	If $\omega\subset\Ss'$, then the process is unchanged, and $\Ss''$ is a new sphere-system (that is disjoint from $\Ff''$) as in the proof of Lemma~\ref{lem:innermost}.
	If $\omega\subset\Ff'$, then the result $\Ff''$ of the surgery contains closed components, which we discard.
	
	By construction, $\Ss''$ is disjoint from $\Ff''$, hence is disjoint from $L$.
	We claim that each component of $Y\setminus\nu(\Ss'')$ contains at most one component of $L$.
	This follows inductively: $\Ss''$ begins by isolating components since it is a copy of $\Ss'$, which isolates components of $L$.
	The sphere-system $\Ss''$ is unchanged when $\omega\subset\Ff''$, and the sphere-system $\Ss''$ changes in exactly the same way as in Lemma~\ref{lem:innermost} when $\omega\subset\Ss''$: either the complement of $\Ss''$ is split along a neatly-embedded disk, or a 3--dimensional 2--handle is attached the boundary of the complement, as in Figure~\ref{fig:surgery_lemma}(\textsc{a}).
	Each option preserves the fact that a sphere-system isolates components of $L$; hence $\Ss''$ has this property.

	By Lemma~\ref{lem:new_ball-system}, $\Ss''$ bounds ball-systems $\Pp$ and $\Pp'$ in $X$ and $X'$, respectively, that are disjoint from $\Bb$ and $\Bb'$.
	In the present setting, we must also produce invariant, almost-disjoint bridge-balls $\Delta$ and $\Delta'$ for $\Dd$ and $\Dd'$, respectively, that are disjoint from $\Pp$ and $\Pp'$.
	Following the proof of Lemma~\ref{lem:new_ball-system}, if $Z$ contains a disk $D$ with $\partial D = U$, then we note that, by Lemma~\ref{lem:disks_in_ball_diffeomorphic}, the pair $(Z, \Dd)$ is diffeomorphic rel-boundary to the cone on $(\partial Z, U)$.
	Take $\Delta$ to be the cone on $\Ff''$ with respect to this cone structure.
	At each level, since $\Ss''_Z$ is disjoint from $\Ff''_Z$, and the union of all disks in $\Ff''_Z$ is connected, we can choose the innermost sphere $S_I$ that we cap off to be a sphere that bounds a ball $B_I$ that is disjoint from the other spheres and from $\Ff''_Z$.
	The entirety of tubing and capping process of Lemma~\ref{lem:new_ball-system} can now be applied, with this alteration, and the ball-system produced will be disjoint from $\Delta$ since $\Delta$ is just the cone on $\Ff''_Z$.
	This gives ball systems for $X$ and $X'$, $\Pp$ and $\Pp'$ respectively, and bridge-ball systems, $\Delta$ and $\Delta'$ repectively, which are disjoint from their respective ball systems.
	
	Finally, we claim that each ball of $X\setminus \Pp$ and $X\setminus \Pp'$ contains at most one component of $\Dd$ and $\Dd'$, respectively.
	This follows since $\partial \Pp = \partial \Pp' = \Ss''$ isolates components of $\partial \Dd = \partial \Dd' = L$.

	We have that $\Pp\cup\Delta$ and $\Pp'\cup\Delta'$ are linear parting pairs of ball-systems and bridge-balls for $(X, \Dd)$ and $(X', \Dd')$, respectively, such that $\partial\Pp = \partial\Pp'$.
	As in the proof of Theorem~\ref{thm:laud_poen}, it follows that $X$ and $X'$ have isomorphic handle-decompositions, hence are equivariantly diffeomorphic by Lemmata~\ref{lem:unique_handle_attachment} and Corollary~\ref{cor:disk_fillings_diffeomorphic}, with the former applying to the 3--handles, and the latter applying to the 4--handles, which may contain disks.
\end{proof}

\section{Open Questions}
\label{sec:open_questions}

We conclude with a discussion of several fascinating open questions that serve to motivate the content of the present paper, starting with questions about \emph{linearization} of cyclic group actions.
In dimension three, the~\ref{LT} implies that every cyclic group action on $S^3$ is smoothly equivalent to a linear action.
In dimension four, the resolution in the negative of the Smith Conjecture~\cite{Gif_66_The-generalized-Smith-conjecture,Gor_74_On-the-higher-dimensional-Smith} shows that there are cyclic group actions on $S^4$ whose fixed-point set is a non-trivial $2$--knot; hence, these actions cannot be linear.
The knot type of the fixed-point set is, to our knowledge, the only known invariant of two-sphere-fixing cyclic actions on $S^4$, naturally leading to the following conjecture.

\begin{conjecture}[Unknot Linearization]
\label{conj:unknot_linear}
	If $\mathbb Z_n$ acts smoothly on $S^4$ and the fixed-point set is an \emph{unknotted} two-sphere, then the action is smoothly equivalent to a linear action.
\end{conjecture}

This conjecture has deep connections with fundamental questions about smooth structures on $4$--manifolds:
Given actions $\rho$ and $\rho'$ of $\mathbb Z_n$ on $S^4$ inducing branched covers with fixed-point sets $S$ and $S'$, the actions are smoothly equivalent if and only if the quotients $(X/\rho, S/\rho)$ and $(X/\rho', S'/\rho)$ are diffeomorphic as pairs.
Letting $\ell$ be a linear action on $S^4$ fixing an unknotted two-sphere $S$, and letting $\rho$ be some candidate action for Conjecture~\ref{conj:unknot_linear} also fixing $S$, then $\rho$ is linear if and only if $(S^4/\ell, S/\ell)$ is diffeomorphic to $(S^4/\rho, S/\rho)$.
Clearly $(S^4/\ell, S/\ell)$ is diffeomorphic to $(S^4, S)$.
Thus, $\rho$ can fail to be smoothly equivalent to a linear action either by having $S^4/\rho$ fail to be $S^4$, or having $S/\rho$ fail to be diffeomorphic to $S$.

\begin{proposition}
\label{prop:counterexample}
	Suppose that $\rho$ is a non-linear action of $\Z_n$ on $S^4$ that fixes the unknotted two-sphere $S$.
	Then, either $S^4/\rho$ is an exotic 4--sphere, or $S/\rho$ is a non-trivial 2--knot with group $\Z$.
\end{proposition}

\begin{proof}
Relatively simple algebraic topology shows that $\pi_1(S^4/\rho\setminus S/\rho) = \mathbb Z$ and that $S^4/\rho$ is a homotopy $S^4$: we know that the $\Z_n$--action on $\partial(S^4\setminus \nu(S)) = S^1\times S^2$ is rotation in the first coordinate, which determines the map on $\pi_1$ from $S^1\times S^2$ to $(S^1\times S^2)/\rho$ defining the covering. 
This, combined with the fact that the inclusion of $\partial(S^4\setminus \nu(S))$ into $S^4\setminus \nu(S)$ is a $\pi_1$--isomorphism, determines the covering map onto $(S^4\setminus \nu(S))/\rho$ defining the covering and hence determines that $\pi_1(S^4/\rho\setminus S/\rho) = \mathbb Z$. 
Attaching $\nu(S)/\rho$ now amounts to a $2$--handle attachment along $S^1\times\{pt\}$ and a $4$--handle attachment, giving that $S^4/\rho$ is a homotopy $S^4$.
\end{proof}

For a discussion relating to the question of whether $S/\rho$ can be a non-trivial 2--knot, see Daniel Hartman's MathOverflow post~\cite{htt__A-conjecture-about-homotopy}.
By work of Freedman~\cite{Fre_82_The-topology-of-four-dimensional, FreQui_90_Topology-of-4-manifolds, Beh_21_The-Disc-Embedding-Theorem}, this conjecture is true in the topological category, so a counterexample would be nonlinear and smoothly exotic.
A related question, which arose through conversation with Daniel Hartman related to his MathOverflow post, is the following.

\begin{question}
\label{ques:s1xb3}
	Does every cyclic action on $S^1\times B^3$ admit an equivariant, neatly embedded, non-separating $B^3$?
\end{question}

This is subtly different than asking if all actions on $S^1\times B^3$ are parted (see Section~\ref{sec:equivariant_handlebodies}), since the orbit of these balls do not need to have $B^4$ complements.
The orbit of such a $B^3$ gives an invariant ``handle-decomposition'' of $S^1\times B^3$ into $1$--handles and (possibly exotic) homotopy $4$--ball ``$0$--handles".
(Again, the linearity hypothesis is missing.)
The connection with the above is that an action with no parting $B^3$ that is \emph{free} could be completed by attaching a cancelling $2$--handle into an action on $S^4$ fixing an unknotted two-sphere, but which is not smoothly equivalent to the linear action fixing this unknot.
A counterexample arising in this way would be such that the quotient knot $S/\rho$ is not unknotted, since a $3$--ball realizing the unknotting of $S/\rho$ could be lifted to an equivariant parting $B^3$.
Hence, the quotient of such an action on the uncompleted $S^1\times B^3$ would be a manifold with infinite cyclic fundamental group and boundary $S^1\times S^2$, which does not admit any neatly embedded non-seperating $B^3$ (i.e. an exotic $S^1\times B^3$).
This question can be viewed as a weak form of a generalization to four dimensions of the equivariant disk theorem in dimension three. 

This question is also interesting for non-free cyclic actions: A result of Gabbard shows that there exists a $\Z_2$--action on $S^1\times B^3$ restricting on the boundary $S^1\times S^2$ to factor-wise reflection in each component that does not admit an \emph{invariant} parting $3$--ball.
See~\cite[Theorem 5.8]{Gab_24_Equivariant-double-slice-genus}; the $S^1\times B^3$ in question is the exterior of Gabbard's 2--knot $(J,\tau)$.

There is no clear equivariant $3$--ball parting this action and no clear way to obstruct such a $3$--ball.
Nevertheless, if such an equivariant $3$--ball existed, its orbit gives a $3$--sphere with extremely strange properties in the vein of Proposition~\ref{prop:counterexample}, and so, with Gabbard, we conjecture the following.

\begin{conjecture}
\label{conj:Gabbard}
	The strongly invertible $2$--knot $(J,\tau)$ of~\cite[Theorem 5.8]{Gab_24_Equivariant-double-slice-genus} does not bound an equivariant $3$--ball.
\end{conjecture}

Rather than fixing a two-sphere, a cyclic action on $S^4$ can fix a pair of points.

\begin{conjecture}[Point-Fixing Linearization]
\label{conj:pf_linear}
	If $\mathbb Z_n$ acts on $S^4$ smoothly with exactly two fixed points, then the action is smoothly equivalent to a linear action.
\end{conjecture}

This conjecture has been studied extensively, see~\cite[Conjecture~2.2]{Che_10_Group-actions-on-4-manifolds:} and admits a nice reformulation in terms of showing the triviality of a related $s$--cobordism of lens spaces $S^3/\mathbb Z_n$.

Both this conjecture and Conjecture~\ref{conj:unknot_linear} are equivalent to corresponding conjectures for cyclic actions on $B^4$ via the removal of a tubular neighborhood of a fixed point.
Progress on these conjectures would advance our understanding of group actions on $4$--manifolds up to smooth equivalence by removing the possibility that one could change an action on $X$ without changing the smooth topology of $X$ or the fixed-point set by connect-summing with a non-trivial action on $S^4$.

These conjectures are also the first stone on the path of the following ambitious question.

\begin{question}
\label{ques:action_invariant}
	Is the smooth equivalence class of a $\Z_p$ action $\rho$ on $S^4$ determined by the knot type of $\text{Fix}(\rho)$?
	In other words, is the map $\rho\mapsto\Fix(\rho)$ an injection from the set of actions up to smooth equivalence to the set of even-dimensional knots ($S^2$--knots or $S^0$--knots) in $S^4$ up to diffeomorphism of pairs?
\end{question}

\bibliographystyle{amsalpha}
\bibliography{equiv_LP.bib}

\end{document}